\documentclass[a4paper,11pt,english]{article}
\usepackage{amsmath,amssymb,enumerate,color}
\usepackage{amsthm}
\usepackage{calc,txfonts}

\newcommand{\R}{\mathbb{R}}

\newcommand{\N}{\mathbb{N}}
\newcommand{\ep}{\varepsilon}
\newcommand{\pa}{\partial}

\newcommand{\lr}[1]{{}\langle{}#1{}\rangle{}}

\newtheorem{theorem}{Theorem}[section]
\newtheorem{lemma}[theorem]{Lemma}
\newtheorem{proposition}[theorem]{Proposition}
\newtheorem{corollary}[theorem]{Corollary}
\theoremstyle{remark}
\newtheorem{remark}{Remark}[section]
\theoremstyle{definition}

\newtheorem{definition}{Definition}[section]

\numberwithin{equation}{section}

\makeatletter
\def\@cite#1#2{[{{\bfseries #1}\if@tempswa , #2\fi}]}
\makeatother                  
\begin{document}
\begin{center}
\Large{{\bf
Weighted energy estimates for wave equation 
with space-dependent damping term 
for slowly decaying initial data
}}
\end{center}

\vspace{5pt}

\begin{center}
Motohiro Sobajima%
\footnote{
Department of Mathematics, 
Faculty of Science and Technology, Tokyo University of Science,  
2641 Yamazaki, Noda-shi, Chiba, 278-8510, Japan,  
E-mail:\ {\tt msobajima1984@gmail.com}}
and 
Yuta Wakasugi%
\footnote{
Graduate School of Science and Engineering, Ehime University, 
3, Bunkyo-cho, Matsuyama, Ehime, 790-8577, Japan, 
E-mail:\ {\tt wakasugi.yuta.vi@ehime-u.ac.jp}.}
\end{center}

\newenvironment{summary}{\vspace{.5\baselineskip}\begin{list}{}{%
     \setlength{\baselineskip}{0.85\baselineskip}
     \setlength{\topsep}{0pt}
     \setlength{\leftmargin}{12mm}
     \setlength{\rightmargin}{12mm}
     \setlength{\listparindent}{0mm}
     \setlength{\itemindent}{\listparindent}
     \setlength{\parsep}{0pt}
     \item\relax}}{\end{list}\vspace{.5\baselineskip}}
\begin{summary}
{\footnotesize {\bf Abstract.}
This paper is concerned with weighted energy estimates 
for solutions to wave equation 
$\pa_t^2u-\Delta u + a(x)\pa_tu=0$
with space-dependent damping term $a(x)=|x|^{-\alpha}$ 
$(\alpha\in [0,1))$
in an exterior domain $\Omega$ having a smooth boundary. 
The main result asserts that 
the weighted energy estimates with 
weight function like polymonials  
are given and these decay rate are almost sharp, 
even when 
the initial data do not have compact support in $\Omega$. 
The crucial idea is to use special solution of 
$\pa_tu=|x|^\alpha\Delta u$ including Kummer's confluent hypergeometric functions.
}
\end{summary}



\section{Introduction}

In this paper we consider the wave equation 
with space-dependent damping term 
\begin{equation}\label{dw}
\begin{cases}
\pa_t^2u-\Delta u + a(x)\pa_tu=0, & x\in \Omega,\ t>0, 
\\
u(x,t)=0, & x\in\pa\Omega, \ t>0,
\\
u(x,0)=u_0(x), \quad \pa_tu(x,0)=u_1(x), & x\in \Omega, 
\end{cases}
\end{equation}
where $a(x)=|x|^{-\alpha}$
with a parameter $\alpha\geq 0$,
$\Omega$ is an exterior domain in $\R^N$ $(N\geq 2)$ 
with smooth boundary and satisfies $0\notin \overline{\Omega}$
and 
the initial data
$(u_0, u_1)$ satisfy the compatibility condition of 
order $1$, that is, 
\begin{align}
\label{ini}
	(u_0, u_1) \in (H^2(\Omega) \cap H^1_0(\Omega) ) \times H^1_0(\Omega).
\end{align}
If $\alpha=0$, then \eqref{dw} becomes the usual 
damped wave equation. Although we can take 
$\Omega=\R^N$, we skip the case of whole space 
for the simplicity of terminology.
The term
$a(x)\pa_tu$
describes the damping effect,
which plays a role in reducing the energy of the wave.
We remark that the coefficient $a(x)$ of the damping term 
is uniformly bounded in $\overline{\Omega}$ 
and therefore it is well-known that \eqref{dw} has a unique 
solution $u$ in the following class:
\begin{align}
\label{sol}
u\in C^2([0,\infty);L^2(\Omega))
\cap
C^1([0,\infty);H^1_0(\Omega))
\cap
C([0,\infty);H^2(\Omega) \cap H^1_0(\Omega))
\end{align}
(see Ikawa \cite[Theorem 2]{Ik68}).

Our purpose of this paper is to establish 
the weighted energy estimates and the asymptotic behavior 
of solutions to \eqref{dw} 
without assuming the compactness of the support of initial data.

Matsumura \cite{Ma76} proved that if $\Omega=\R^N$ and $a(x)\equiv 1$, 
then the solution $u$ of \eqref{dw} satisfies 
the energy decay estimate
\[
\int_{\R^N}\Big(|\nabla u(x,t)|^2+|\pa_tu(x,t)|^2\Big)\,dx
\leq 
C(1+t)^{-\frac{N}{2}-1}\|(u_0,u_1)\|_{(H^1\cap L^1)\times (L^2\cap L^1)}^2.
\]
Later on,
it is shown in \cite{YaMi00, Kar00, MaNi03, Nishihara03, HoOg04, Na04} that
$u$ is asymptotically approximated by 
the one of the problem
\begin{align}
\label{heat0}
	\left\{\begin{array}{ll}
	   \pa_tv-\Delta v= 0,& x\in \R^N,\ t>0,
\\
	v(x,0) = u_0(x)+u_1(x),&x\in \R^N.
	\end{array}\right.
\end{align}
In particular, we have
\[
\|u(\cdot,t)-v(\cdot,t)\|_{L^2}=o(t^{-\frac{N}{4}})
\]
as $t\to\infty$.
This is called the diffusion phenomena and studied by several researchers
including exterior domain cases
\cite{Ik02, ChHa03}.
Matsumura \cite{Ma77} also dealt with the
energy decay of solutions to \eqref{dw}
for general cases with $a(x)\geq \lr{x}^{-\alpha}
=(1+|x|^2)^{-\frac{\alpha}{2}}$ $0\leq \alpha\leq 1$. 
On the other hand, Mochizuki \cite{Mo76} 
showed that 
if $0\leq a(x)\leq C\lr{x}^{-\alpha}$ for some $\alpha>1$, 
then in general, the energy of the solution to \eqref{dw} 
does not vanish as $t\to \infty$.
Moreover, the solution is asymptotically equivalent with
the free wave equation.
We remark that they actually treated
the case the damping coefficient also depends on $t$.
These works clarify that the 
threshold of diffusion phenomena is $\alpha=1$.

After that certain decay estimates for weighted energy
\begin{align*}
\int_{\Omega}\Big(|\nabla u|^2+|\pa_tu|^2\Big)
\exp\left(\frac{c_{\delta}|x|^{2-\alpha}}{1+t}\right)\,dx
\leq 
C_\delta(1+t)^{-\frac{N-\alpha}{2-\alpha}-1+\delta}
\int_{\Omega} (|u_0|^2 + |\nabla u_0|^2 + |u_1|^2)
\exp\left( c_{\delta}|x|^{2-\alpha} \right) \,dx
\end{align*}
(with any $\delta>0$ and some $c_\delta,C_\delta>0$)
have been proved by 
Ikehata \cite{Ik05IJPAM} (without compactness of support of initial data)
and Todorova--Yordanov \cite{ToYo09} 
and Radu-Todorova--Yordanov \cite{RTY10} 
when $a(x)$ is radially symmetric. 
Similar weighted energy estimates for higher derivatives of solutions 
are also shown in Radu--Todorova--Yordanov \cite{RTY09}.  
In the case $\Omega=\R^N$ and 
$a(x)=\lr{x}^{-\alpha}$,
the second author proved in \cite{Wa14} that 
the solution $u$ of \eqref{dw} 
has the same asymptotic behavior 
as the one of the following parabolic problem
\begin{align}
\label{heat}
	\left\{\begin{array}{ll}
	   \pa_tv-a(x)^{-1}\Delta v= 0,& x\in \R^N,\ t>0,
\\
	v(x,0) = u_0(x)+a(x)^{-1}u_1(x),&x\in \R^N.
	\end{array}\right.
\end{align}
Then in \cite{So_Wa1,So_Wa2} 
the problem \eqref{dw} in an exterior domain
with non-radially symmetric damping terms satisfying
\[
\lim_{|x|\to\infty}\Big(|x|^{\alpha}a(x)\Big)=a_0>0
\]
could be considered and 
it is shown that the asymptotic behavior 
of solutions to \eqref{dw}
can be also given by the solution of \eqref{heat}, 
however,  
only when the initial data are compactly supported.  

We would summarize that if the initial data 
are not compactly supported, then a kind of 
weighted energy estimates is quite few; 
note that Ikehata gave one of weighted energy estimates in \cite{Ik05IJPAM} but the initial data 
are required to have an exponential decay.
The study of asymptotic behavior of solutions 
seems difficult to treat without weighted energy estimates. 
	
The first purpose of this paper 
is to establish a weighted energy estimates for solutions to \eqref{dw}
with a typical damping $a(x)=|x|^{-\alpha}$, which can be 
applied to initial data with polynomial decay.
The second is to find the asymptotic behavior 
of solutions to \eqref{dw} as a solution of \eqref{heat}
as an application of weighted energy estimates obtained in the first part.

Now we are in a position to state our first result.

\begin{theorem}\label{main1}
Let $u$ be a unique solution of \eqref{dw} with initial data 
$(u_0,u_1)\in (H^2\cap H^1_0(\Omega))\times H^1_0(\Omega)$.
Assume that $N\geq 2$ and there exists $\gamma\in [\alpha,N+2-2\alpha)$ such that 
\begin{equation}\label{ass.thm}
\int_{\Omega}
   \Big(|\nabla u_0(x)|^2+(u_1(x))^2\Big)|x|^{\gamma}
\,dx
<\infty. 
\end{equation}
Then there exists a constant $C>0$ such that 
\begin{align}
\nonumber
&
\int_{\Omega}
   \Big(|\nabla u(x,t)|^2+(\pa_tu(x,t))^2\Big)
(t+|x|^{2-\alpha})^{\frac{\gamma}{2-\alpha}}
\,dx
+
\int_{0}^{t}
\left(
\int_{\Omega}
   |x|^{-\alpha}|u(x,s)|^2
   (s+|x|^{2-\alpha})^{\frac{\gamma}{2-\alpha}}\,dx
\right)
\,ds
\\
\label{w.est}
&\leq
\begin{cases} 
\displaystyle
C
\int_{\Omega}
   \Big(|\nabla u_0(x)|^2+(u_1(x))^2\Big)|x|^{\gamma}
\,dx
+
\int_{\Omega}
   |u_0(x)|^2|x|^{-\alpha}
\,dx
&
\textrm{ if }\gamma< 2-\alpha,
\\[10pt]
\displaystyle
C
\int_{\Omega}
   \Big(|\nabla u_0(x)|^2+(u_1(x))^2\Big)|x|^{\gamma}
\,dx
&\text{otherwise}.
\end{cases}
\end{align}
\end{theorem}

\begin{remark}
In the case $\gamma=2-\alpha$, 
we can use the weighted Hardy inequality
\[
\left(\frac{N-\alpha}{2}\right)^2\int_{\Omega}|u|^2|x|^{-\alpha}\,dx
\leq 
\int_{\Omega}|\nabla u|^2|x|^{2-\alpha}\,dx
\]
which implies that \eqref{w.est} can be regarded as an estimate 
continuously depending on $\alpha$ 
(see the proof of Proposition \ref{en.decay0} in page \pageref{en.decay0}).
\end{remark}

\begin{remark}
If $\alpha=0$ and $\gamma\leq 2$, 
then the assertion 
of Theorem \ref{main1} does not have novelty. 
Indeed, \cite{Ikehata00} proved that 
if $\alpha=0$, then 
\begin{equation*}
\int_{\Omega}
   \Big(|\nabla u(x,t)|^2+|\partial_tu(x,t)|^2\Big)
\,dx
= O(t^{-2})
\end{equation*}
as $t\to \infty$ by only assuming $\nabla u_0, |x|^2(u_0+u_1)\in L^2(\Omega)$. 
Therefore Theorem \ref{main1} is meaningful 
when either $\alpha\in (0,1)$ or $\gamma>2$ is satisfied.
\end{remark}

\begin{remark}
Here we point out our basic idea of the choice of weight functions 
in the energy functional. If $\Phi\in C^2(\R^N\times [0,\infty))$ is a 
positive function,  
then by putting $w=\Phi^{-1}v$, we can formally compute the following 
weighted estimate for \eqref{heat}:
\begin{align*}
\frac{d}{dt}\int_{\Omega}|v|^2a(x)\Phi^{-1}\,dx
&=
2\int_{\Omega}v\pa_tva(x)\Phi^{-1}\,dx
-
\int_{\Omega}|v|^2a(x)\Phi^{-2}\pa_t\Phi\,dx
\\
&=
2\int_{\Omega}v\Delta v\Phi^{-1}\,dx
-
\int_{\Omega}|v|^2\Phi^{-2}\Delta \Phi\,dx
-
\int_{\Omega}|v|^2\Phi^{-2}\Big(a(x)\pa_t\Phi-\Delta \Phi\Big)\,dx
\\
&=
2\int_{\Omega}w\Delta (\Phi w)\,dx
-
\int_{\Omega}|w|^2\Delta \Phi\,dx
-
\int_{\Omega}|w|^2\Big(a(x)\pa_t\Phi-\Delta \Phi\Big)\,dx.
\end{align*}
Integration by parts we have
\[
\frac{d}{dt}\int_{\Omega}|v|^2a(x)\Phi^{-1}\,dx
\leq 
-2\int_{\Omega}|\nabla w|^2\Phi\,dx
-
\int_{\Omega}|w|^2\Big(a(x)\pa_t\Phi-\Delta \Phi\Big)\,dx.
\]
This means that if $\Phi$ is a positive (super-)solution of \eqref{heat}, 
then the value $\int_{\Omega}|v|^2a(x)\Phi^{-1}\,dx$ is decreasing.
Therefore an $L^2$-estimate with weighted measure $a(x)dx$ can be proved directly from 
an $L^\infty$-estimate for $\Phi$. 
In Ikehata \cite{Ik05IJPAM}, Todorova--Yordanov \cite{ToYo09},
\cite{Wa14}, \cite{So_Wa1} and \cite{So_Wa2} 
essentially used functions similar to Gaussian function
\[
\Phi(x,t)=(1+t)^{-\frac{N-\alpha}{2-\alpha}}e^{-\frac{|x|^{2-\alpha}}{(2-\alpha)^2(1+t)}}
\]
which also appears in the analysis of corresponding semilinear problem, 
see e.g., \cite{ToYo01,IT05,LNZ10,Nishihara10,Wa12}.
However, the technique in these previous papers requires that at least 
the initial data 
decays exponentially at spacial infinity. 
To overcome this difficulty, 
we choose a different solution of $\pa_tv=|x|^{\alpha}\Delta v$ 
via a family of self-similar solutions in Section 2.1, which is 
the crucial point in the present paper. 
\end{remark}

Next we consider diffusion phenomena. 
To state the result, we introduce the heat semigroup corresponding to \eqref{heat}
as follows:
\begin{gather*}
L^p_{d\mu}=\left\{f\in L^p_{\rm loc}(\Omega)\;;\;\|f\|_{L^p_{d\mu}}:=\left(\int_{\Omega}|f(x)|^p\,a(x)dx\right)^{\frac{1}{p}}<\infty\right\},
\end{gather*}
and $L_{\min}u=a(x)^{-1}\Delta u$, with domain $D(L_{\min})=\{u\in C^\infty(\R^N)\;\;u|_{\pa\Omega}=0\}$; 
note that the operator $-L_{\min}$ is nonnegative and symmetric in $L^2_{d\mu}$. 
Define $-L_*$ as the Friedrichs extension of $-L_{\min}$
(for the precise definition see Lemma \ref{L*} 
in page \pageref{L*}). 

Then the statement of diffusion phenomena is the following:
\begin{theorem}\label{main3}
Let $(u_0,u_1)$ satisfy the compatibility condition of order $1$. 
Assume that there exists 
$\gamma\in (2-\alpha, N+2-2\alpha)$
such that 
\[
\mathcal{E}_0
:=\int_{\Omega}
   \Big(|\nabla u_0|^{2}+|u_1|^2\Big)|x|^{\gamma}
\,dx<\infty,
\quad 
\mathcal{E}_1
:=\int_{\Omega}
   \Big(|\nabla u_1|^{2}+|u_2|^2\Big)|x|^{\gamma+2}
\,dx<\infty
\]
with $u_2=-\Delta u_0+a(x)u_1$.
Then 
$u_0+|x|^{\alpha}u_1\in L^2_{d\mu}$ and
there exists a constant $C>0$ such that 
\begin{align}\label{dif.phe}
\Big\|u(t)-e^{tL_*}[u_0+|x|^{\alpha}u_1]\Big\|_{L^2_{d\mu}}
\leq C(1+t)^{-\frac{\gamma-\alpha}{2(2-\alpha)}}
(\mathcal{E}_0+\mathcal{E}_1)^{\frac{1}{2}}.
\end{align}
\end{theorem}

\begin{remark}
Under the assumption in Theorem \ref{main3}, we have 
$u_0,|x|^{-\alpha}u_1\in L^p_{d\mu}$ for $p\in (\frac{2(N-\alpha)}{N+\gamma-2},2]$ 
by the simple calculation with H\"older's inequality
\[
\big\||x|^\alpha u_1\big\|_{L^p_{d\mu}}
\leq 
\left(
\int_{\Omega}|u_1|^{2}|x|^{\gamma}\,dx\right)^{\frac{1}{2}}
\left(
\int_{\Omega}|x|^{-N-\frac{N+\gamma -2\alpha}{2-p}(p-\frac{2(N-\alpha)}{N+\gamma -2\alpha})
}\,dx\right)^{\frac{1}{p}-\frac{1}{2}}<\infty.
\]
This gives 
a decay estimate for $e^{tL_*}[u_0+|x|^{\alpha}u_1]$ as
\[
\Big\|e^{tL_*}[u_0+|x|^{\alpha}u_1]\Big\|_{L^2_{d\mu}}
\leq C(1+t)^{-\frac{\gamma-\alpha}{2(2-\alpha)}+\frac{1-\alpha}{2-\alpha}+\ep}
\| u_0+|x|^{-\alpha}u_1 \|_{L^p_{d\mu}}
\]
for sufficiently small $\ep > 0$
(see also Remark \ref{rem2.3} in page \pageref{rem2.3}).
Therefore the estimate \eqref{dif.phe} enables us to determine 
the asymptotic behavior of solution to the problem \eqref{dw} 
with non-compactly supported initial data. 
\end{remark}

Finally, we give a corollary of Theorems \ref{main1} and \ref{main3} 
with decay estimates similar to corresponding heat equation 
\eqref{heat} 
for initial data for a certain class 
which also contains functions behave like polynomials.
\begin{corollary}\label{coro}
Let $(u_0,u_1)$ satisfy the compatibility condition of order $1$. 
Assume that 
\[
\int_{\Omega}
   \Big(|\nabla u_0|^{2}+|u_1|^2\Big)|x|^{N+2-2\alpha}
\,dx<\infty.
\]
Then for every $\ep>0$,  
\[
\int_{\Omega}
   \Big(|\nabla u(x,t)|^2+(\pa_tu(x,t))^2\Big)
(t+|x|^{2-\alpha})^{\frac{N-\alpha}{2-\alpha}+1-\frac{\ep}{2-\alpha}}
\,dx
=O(1)
\]
as $t\to \infty$. 
Moreover, further assume that 
\[
\int_{\Omega}
   \Big(|\nabla u_1|^{2}+\big|\Delta u_0-a(x)u_1\big|^2\Big)|x|^{N-2\alpha}
\,dx<\infty.
\]
Then for every $\ep>0$,
\begin{align*}
\Big\|u(t)-e^{tL_*}[u_0+|x|^{\alpha}u_1]\Big\|_{L^2_{d\mu}}
&=
O(t^{-\frac{N-\alpha}{2(2-\alpha)}-\frac{1-\alpha}{2-\alpha}+\ep}), 
\\
\Big\|e^{tL_*}[u_0+|x|^{\alpha}u_1]\Big\|_{L^2_{d\mu}}
&=
O(t^{-\frac{N-\alpha}{2(2-\alpha)}+\ep})
\end{align*}
as $t\to \infty$. 
\end{corollary}

This paper is organized as follows. 
In Section 2, we construct a 
suitable weight function from 
a self-similar solution of \eqref{heat} 
with a parameter, which is 
quite different from that in 
\cite{Ik05IJPAM}, \cite{ToYo09}, \cite{RTY10}, \cite{Wa14}, \cite{So_Wa1} 
and \cite{So_Wa2} as mentioned before. 
We also mention the properties of the semigroup generated 
by $L_*=a(x)\Delta$.
In Section 3, the weighted energy estimates 
for solutions to \eqref{dw} are proved. 
Section 4 is devoted to show ones for higher derivatives. 
Finally, 
the asymptotic behavior of solutions 
to \eqref{dw} is given in Section 5.

\section{Preliminaries}

\subsection{Self-similar solution and Kummer's confluent hypergeometric function}
To construct a suitable weight function, 
we start with a construction of radially 
symmetric self-similar solutions of the following 
heat equation
\begin{equation}\label{heat}
\pa_t \Phi = |x|^{\alpha} \Delta \Phi, \quad x\in \R^N, \ t>0.
\end{equation}
If $\Phi$ is a solution of \eqref{heat}. Then we easily see that 
\[
\lambda^{-\widetilde{\beta}}\Phi(\lambda x,\lambda^{2-\alpha}t)
\]
with $\widetilde{\beta} \in \R$
is also a solution of \eqref{heat}.
The following is the characterization of radially symmetric 
self-similar solutions. 

\begin{lemma}\label{self.sim}
Let $\widetilde{\beta}\in\R$. A radial solution $\Phi$ of \eqref{heat} satisfies 
\begin{equation}\label{scaling}
\Phi(x,t)=\lambda^{-\widetilde{\beta}}\Phi(\lambda x,\lambda^{2-\alpha}t), 
\quad 
x\in \R^N\setminus\{0\}, t>0
\end{equation}
for every $\lambda>0$ if and only if 
\begin{equation}\label{sol.form}
\Phi(x,t)=
t^{\frac{\widetilde{\beta}}{2-\alpha}}
\varphi\left(\frac{|x|^{2-\alpha}}{(2-\alpha)^2t}\right),
\quad 
x\in \R^N\setminus\{0\}, t>0
\end{equation}
where $\varphi\in C^\infty((0,\infty))$ satisfies
\begin{align}
\label{eq:varphi1}
s\varphi''(s)
+\left(\frac{N-\alpha}{2-\alpha}+s\right)\varphi'(s)
-\frac{\widetilde{\beta}}{2-\alpha}\varphi(s)
=0, \quad s>0.
\end{align}
\end{lemma}
\begin{proof}
By direct calculation, we can verify that 
the function $t^{\frac{\widetilde{\beta}}{2-\alpha}}
\varphi\big(\frac{|x|^{2-\alpha}}{(2-\alpha)^2t}\big)$
satisfies both \eqref{heat} and \eqref{scaling}
if $\varphi$ satisfies \eqref{eq:varphi1}.
Conversely, let $\Phi$ satisfy \eqref{heat} and \eqref{scaling}. Then 
choosing $\lambda=((2-\alpha)^2t)^{-\frac{1}{2-\alpha}}$,  we see by \eqref{scaling} that 
\[
\Phi(x,t)=
((2-\alpha)^2t)^{\frac{\widetilde{\beta}}{2-\alpha}}
\Phi\left(((2-\alpha)^2t)^{-\frac{1}{2-\alpha}}x,(2-\alpha)^{-2}\right),
\quad 
x\in \R^N\setminus\{0\}, t>0.
\] 
Therefore 
noting that $\Phi$ is radial, by taking 
\[
\varphi(s)=
(2-\alpha)^{\frac{2\widetilde{\beta}}{2-\alpha}}
\Phi\left(s^{\frac{1}{2-\alpha}}\frac{x}{|x|}, (2-\alpha)^{-2}\right), 
\quad 
x\in \R^N\setminus\{0\}, t>0,
\]
we have \eqref{sol.form} and \eqref{eq:varphi1}.
\end{proof}

\begin{lemma}\label{varphi}
Let $\beta\in \R$. 
Assume that $\varphi$ satisfies  
\begin{equation}\label{eq:varphi}
s\varphi''(s)+\Big(\frac{N-\alpha}{2-\alpha}+s\Big)\varphi'(s)
+\beta\varphi(s)=0,  \quad s>0
\end{equation}
with $\lim_{s\downarrow0}\varphi(s)=1$. 
Then 
\[
\varphi(s)=e^{-s}
M\left(\frac{N-\alpha}{2-\alpha}-\beta,\frac{N-\alpha}{2-\alpha};s\right),
\]
where $M(\cdot,\cdot;\cdot)$ is the Kummer's confluent hypergeometric function 
(see Definition \ref{def.kummer}).
\end{lemma}
\begin{proof}
Taking $\psi(s)=e^{s}\varphi(s)$, we see from the direct calculation that 
$\psi$ satisfies Kummer's confluent hypergeometric differential equation
\[
s\psi''(s)+\left(\frac{N-\alpha}{2-\alpha}-s\right)\psi'(s)
-\left(\frac{N-\alpha}{2-\alpha}-\beta\right)\psi(s)=0, \quad s>0.
\]
Therefore Lemma \ref{Kummer} {\bf (i)} yields that 
$\psi(s)$ can be given by 
\[
\psi(s)
=k_1M\left(\frac{N-\alpha}{2-\alpha}-\beta,\frac{N-\alpha}{2-\alpha};s\right)
+k_2U\left(\frac{N-\alpha}{2-\alpha}-\beta,\frac{N-\alpha}{2-\alpha};s\right), \quad s>0.
\]
For some $k_1, k_2\in\R$. 
Since $\psi(s)$ and $M(b,c;s)$ converge to $1$ as $s\to 0$ and $U(b,c;s)$ is unbounded at $s=0$, 
we have $k_2=0$ and then
\[
1=\lim_{s\to 0}\psi(s)=k_1\lim_{s\to 0}M\left(\frac{N-\alpha}{2-\alpha}-\beta,\frac{N-\alpha}{2-\alpha};s\right)=k_1. 
\] 
By the definition of $\psi$, the proof is complete.
\end{proof}

Let us fix the notation of concrete functions which we use later. 

\begin{definition}\label{phi.beta}
For $\beta\in \R$, define 
\[
\varphi_\beta(s)=e^{-s}
M\left(\frac{N-\alpha}{2-\alpha}-\beta,\frac{N-\alpha}{2-\alpha};s\right), \quad s\geq 0.
\]
\end{definition}

\begin{lemma}[Properties of $\varphi_\beta$]\label{phi.beta.lem}
The following assertions hold\/$:$
\begin{itemize}
\item[{\bf (i)}] For every $\beta\in \R$, 
\begin{equation}\label{eq:phi.beta}
s\varphi_\beta''(s)+
\left(\frac{N-\alpha}{2-\alpha}+s\right)\varphi_\beta'(s)
+\beta\varphi_\beta(s)=0, \quad s>0.
\end{equation}
\item[{\bf (ii)}] For every $\beta\in \R$, 
\[
\beta\varphi_\beta(s)+s\varphi_\beta'(s)=\beta\varphi_{\beta+1}(s), \quad s>0.
\]
\item[{\bf (iii)}] For every $\beta\in \R$, there exists $C_\beta>0$ such that 
\[
|\varphi_\beta(s)|\leq C_\beta(1+s)^{-\beta},
\quad s\geq0.
\]
\item[{\bf (iv)}] For every $\beta<\frac{N-\alpha}{2-\alpha}$, there exists $c_\beta>0$ such that 
\[
\varphi_\beta(s)\geq c_\beta(1+s)^{-\beta},
\quad s\geq 0.
\]
\end{itemize}
\end{lemma}
\begin{proof}
{\bf (i)}\ The assertion is directly verified by 
Definition \ref{phi.beta} and Lemma \ref{varphi}.

\medskip 

\noindent{\bf (ii)}\ Set $\widetilde{\varphi}(s):=s\varphi_\beta'(s)+\beta\varphi_\beta(s)$ for $s\geq 0$.
Then we see from \eqref{eq:phi.beta} that 
\begin{align*}
\widetilde{\varphi}'(s)
&=s\varphi_\beta''(s)+(\beta+1)\varphi_\beta'(s)
\\
&=
-\left(\frac{N-\alpha}{2-\alpha}+s\right)\varphi_\beta'(s)
-\beta\varphi_\beta(s)
+(\beta+1)\varphi_\beta'(s)
\\
&=
-\left(\frac{N-\alpha}{2-\alpha}-\beta-1\right)\varphi_\beta'(s)
-\widetilde{\varphi}(s)
\end{align*}
and therefore
\begin{align*}
s\widetilde{\varphi}''(s)
+s\widetilde{\varphi}'(s)
&=
-\left(\frac{N-\alpha}{2-\alpha}-\beta-1\right)s\varphi_\beta''(s)
\\
&=
\left(\frac{N-\alpha}{2-\alpha}-\beta-1\right)
\left(\left(\frac{N-\alpha}{2-\alpha}+s\right)\varphi_\beta'(s)+\beta\varphi_\beta(s)\right)
\\
&=
\frac{N-\alpha}{2-\alpha}
\left(\frac{N-\alpha}{2-\alpha}-\beta-1\right)
\varphi_\beta'(s)
+\left(\frac{N-\alpha}{2-\alpha}-\beta-1\right)\widetilde{\varphi}(s)
\\
&=
-\frac{N-\alpha}{2-\alpha}
\left(\widetilde{\varphi}'(s)+\widetilde{\varphi}(s)\right)
+\left(\frac{N-\alpha}{2-\alpha}-\beta-1\right)\widetilde{\varphi}(s)
\\
&=
-\frac{N-\alpha}{2-\alpha}\widetilde{\varphi}'(s)
-(\beta+1)\widetilde{\varphi}(s).
\end{align*}
This means that $\widetilde{\varphi}_{\beta}$ satisfies 
with \eqref{eq:varphi} with $\beta$ replaced with $\beta+1$. 
Noting that $\varphi_{\beta}'(0)=
1-\frac{(2-\alpha)\beta}{N-\alpha}$, we have
$\lim_{s\to 0}\widetilde{\varphi}(s)=\beta$ and therefore 
by Lemma \ref{varphi} we deduce $\widetilde{\varphi}(s)=\beta\varphi_{\beta+1}(s)$ for all $s\geq 0$. 

\medskip 

\noindent{\bf (iii)}.\ By Lemma \ref{Kummer} {\bf (ii)}, we see that 
there exists $R_\beta>0$ such that 
for every $s\geq R_\beta$, 
\[
\frac{1}{2}\,\frac{\Gamma(\frac{N-\alpha}{2-\alpha})}{\Gamma (\frac{N-\alpha}{2-\alpha}-\beta)}
\leq \frac{M(\frac{N-\alpha}{2-\alpha}-\beta,\frac{N-\alpha}{2-\alpha};s)}{s^{-\beta}e^s}
\leq 
\frac{3}{2}\frac{\Gamma(\frac{N-\alpha}{2-\alpha})}{\Gamma (\frac{N-\alpha}{2-\alpha}-\beta)},
\]
or equivalently, for every $s\geq R_\beta$, 
\begin{equation}\label{up.low}
\frac{1}{2}\,\frac{\Gamma(\frac{N-\alpha}{2-\alpha})}{\Gamma (\frac{N-\alpha}{2-\alpha}-\beta)}s^{-\beta}
\leq \varphi_\beta(s)
\leq 
\frac{3}{2}\frac{\Gamma(\frac{N-\alpha}{2-\alpha})}{\Gamma (\frac{N-\alpha}{2-\alpha}-\beta)}s^{-\beta}.
\end{equation}
Since $\varphi_\beta$ is continuous in $[0,R_\beta]$, we verify the assertion. 

\medskip 

\noindent{\bf (iv)} 
Since $\beta<\frac{N-\alpha}{2-\alpha}$ is satisfied,
it follows from 
the definition of $M(b,c;s)$ in Appendix (see page \pageref{def.kummer})
that 
$\varphi_\beta(s)>0$ 
for all $s\geq 0$. From \eqref{up.low} and the positivity of $\varphi_\beta$, 
we have also  lower bounds with the same power $s^{-\beta}$.
\end{proof}

\subsection{Construction of weight functions via $\varphi_\beta$}

Here we define the suitable weight functions for 
weighted energy estimates for solutions of \eqref{dw}. 

\begin{definition}\label{Phi.beta}
For $\beta\in \R$, define 
\[
\Phi_\beta(x,t)=
t^{-\beta}\varphi_\beta\left(\frac{|x|^{2-\alpha}}{(2-\alpha)^2t}\right), \quad x\in \R^N, \ t>0.
\]
\end{definition}

To state the properties of $\Phi_\beta$, we also introduce 
\[
\Psi^{\beta}(x,t)=
\left(t+\frac{|x|^{2-\alpha}}{(2-\alpha)^2}\right)^{\beta}, \quad x\in \R^N, \quad t>0.
\]
\begin{lemma}[Properties of $\Phi_\beta$]
\label{Phi.beta.lem}
The following assertions hold\/$:$
\begin{itemize}
\item[{\bf (i)}] For every $\beta\in \R$, 
\[
\pa_t\Phi_\beta(x,t)=|x|^{\alpha}\Delta\Phi_\beta(x,t), 
\quad x\in \R^N, \ t>0.
\]
\item[{\bf (ii)}] For every $\beta\in \R$, 
\[
\pa_t\Phi_\beta(x,t)=-\beta\Phi_{\beta+1}(x,t)
\quad x\in \R^N, \ t>0.
\]
\item[{\bf (iii)}] For every $\beta\in \R$, 
\[
|\Phi_\beta(x,t)|\leq 
C_\beta\Psi^{-\beta}(x,t),
\quad x\in \R^N, \ t>0.
\]
\item[{\bf (iv)}] For every $\beta<\frac{N-\alpha}{2-\alpha}$, 
\[
\Phi_\beta(x,t)\geq 
c_\beta\Psi^{-\beta}(x,t),
\quad x\in \R^N, \ t>0.
\]
\end{itemize}
\end{lemma}

\begin{proof}
{\bf (iii)} and {\bf (iv)} directly follow from 
the corresponding assertions in Lemma \ref{phi.beta.lem}. 
{\bf (i)} is a consequence of Lemmas \ref{self.sim}, \ref{varphi} 
and Lemma \ref{phi.beta.lem} {\bf (i)}. 
The equality in {\bf (ii)} can be proved by Lemma \ref{phi.beta.lem} {\bf (ii)} as follows:
\begin{align*}
\pa_t\Phi_\beta(x,t)
&=
-\beta t^{-\beta-1}\varphi_\beta\left(\frac{|x|^{2-\alpha}}{(2-\alpha)^2t}\right)
-t^{-\beta}\varphi_\beta'\left(\frac{|x|^{2-\alpha}}{(2-\alpha)^2t}\right)
\frac{|x|^{2-\alpha}}{(2-\alpha)^2t^2}
\\
&=
-
t^{-\beta-1}\left[
\beta \varphi_\beta\left(\frac{|x|^{2-\alpha}}{(2-\alpha)^2t}\right)
+\varphi_\beta'\left(\frac{|x|^{2-\alpha}}{(2-\alpha)^2t}\right)
\frac{|x|^{2-\alpha}}{(2-\alpha)^2t}\right]
\\
&=
-
\beta t^{-\beta-1}
\varphi_{\beta+1}\left(\frac{|x|^{2-\alpha}}{(2-\alpha)^2t}\right)
\\
&=
-
\beta \Phi_{\beta+1}(x,t).
\end{align*}
The proof is complete.
\end{proof}

\begin{remark}
Since the asymptotic behavior of $\varphi_{\beta}(s)$ at $s=\infty$ is explicitly given 
via Lemma \ref{Kummer} {\bf (ii)}, 
we can deduce that for every $x\in \R^N\setminus\{0\}$, 
\[
\lim_{t\downarrow 0}\Phi_\beta(x,t)=\frac{\Gamma(\frac{N-\alpha}{2-\alpha})}{\Gamma(\frac{N-\alpha}{2-\alpha}-\beta)}|x|^{(2-\alpha)\beta}
\]
This means that $\Phi_\beta$ is a solution of \eqref{heat} with 
initial value $\frac{\Gamma(\frac{N-\alpha}{2-\alpha})}{\Gamma(\frac{N-\alpha}{2-\alpha}-\beta)}|x|^{(2-\alpha)\beta}$.
\end{remark}

\begin{remark}\label{rem.equiv}
It follows from Lemma \ref{Phi.beta.lem} {\bf (iii)} and {\bf (iv)} that 
if $\beta<\frac{N-\alpha}{2-\alpha}$, then $\Phi_\beta^{-1}$ and $\Psi_\beta$ 
are equivalent in the sense of weighted functions:
\[
c_\beta^{-1}\Psi^{\beta}(x,t)
\leq 
\Phi_\beta(x,t)^{-1}
\leq 
C_\beta^{-1}\Psi^{\beta}(x,t).
\]
\end{remark}

\subsection{Semigroup generated by $|x|^{\alpha}\Delta$ with Dirichlet boundary condition}

Here we collect the statements of the properties of the semigroup 
generated by $|x|^{\alpha}\Delta$ with Dirichlet boundary condition 
in a weighted $L^2$-space $L^2_{d\mu}$.

First we introduce 
the bilinear form 
\begin{align*}
\begin{cases}
   \mathfrak{a}(u,v)
:=
   \displaystyle
   \int_{\Omega}
      \nabla u(x)\cdot \nabla v(x)
   \,dx, 
\\[8pt]
   D(\mathfrak{a})
:=
   \Big\{
      u\in C_c^\infty(\overline{\Omega})
   \;;\;
      u(x)=0\quad\forall x\in \pa\Omega
   \Big\}
\end{cases}
\end{align*}
in a Hilbert space $L^2_{d\mu}$. 
Then the form $\mathfrak{a}$ is closable, and therefore, 
we denote $\mathfrak{a}_*$ as a closure of $\mathfrak{a}$. 
Then we remark the following four lemmas stated in \cite{So_Wa1}.
\begin{lemma}[{\cite[Lemma 2.1]{So_Wa1}}]
\label{domain.form}
The bilinear form $\mathfrak{a}_*$ can be characterized as follows:
\begin{align}\label{eq:form.dom}
   &D(\mathfrak{a}_*)
=
   \left\{
      u\in L^2_{d\mu}\cap \dot{H}^1(\Omega)
   ;
      \int_{\Omega}
         \frac{\pa u}{\pa x_j}\varphi
      \,dx
   =
      -\int_{\Omega}
         u\frac{\pa \varphi}{\pa x_j}
      \,dx
   \;\;
      \forall \varphi\in C_c^\infty(\R^N)
   \right\},
\\
\label{eq:form}
&\mathfrak{a}_*(u,v)
=
   \int_{\Omega}
      \nabla u(x)\cdot \nabla v(x)
   \,dx.
\end{align}
\end{lemma}

\begin{lemma}[{\cite[Lemma 2.2]{So_Wa1}}]\label{L*}
The operator $-L_*$ in $L^2_{d\mu}$ defined by 
\begin{align*}
   D(L_*)
	&:=
   \Big\{
      u\in D(\mathfrak{a}_*)
   \;;\; 
      \exists f\in L^2_{d\mu}\text{\ s.t.\ }
            \mathfrak{a}_*(u,v)
         =
            (f,v)_{L^2_{d\mu}}
         \quad
            \forall v\in D(\mathfrak{a}_*)
   \Big\},\\
   -L_*u &:=f
\end{align*}
is nonnegative and selfadjoint in $L^2_{d\mu}$. 
Therefore $L_*$ generates an analytic semigroup $e^{tL_*}$ on $L^2_{d\mu}$ 
and satisfies
\[
   \|e^{tL_*}f\|_{L^2_{d\mu}}
\leq 
   \|f\|_{L^2_{d\mu}}, 
\quad 
   \|L_*e^{tL_*}f\|_{L^2_{d\mu}}
\leq 
   \frac{1}{t}\|f\|_{L^2_{d\mu}}, 
\quad 
   \forall\ f\in L^2_{d\mu}.
\]
Furthermore, $L_{*}$ is an extension of $L$ 
defined on $C_c^\infty(\overline{\Omega})$ 
with Dirichlet boundary condition.
\end{lemma}

\begin{lemma}[{\cite[Lemma 2.3]{So_Wa1}}]
\label{domain.op}
We have
\[
\big\{
  u\in H^2(\Omega)\cap H_0^1(\Omega)\;;\;a(x)^{-\frac{1}{2}}\Delta u\in L^2(\Omega)
  \big\}\subset D(L_*)
\] 
and its inclusion is continuous. 
\end{lemma}

\begin{lemma}[{\cite[Proposition 2.6]{So_Wa1}}]
\label{embedding2}
Let $e^{tL_*}$ be given in Lemma \ref{L*}. 
For every $f\in L^2_{d\mu}$, we have
\begin{equation}
\label{2-infty}
	\|e^{tL_*}f\|_{L^\infty}\leq Ct^{-\frac{N-\alpha}{2(2-\alpha)}}\|f\|_{L^2_{d\mu}}.
\end{equation}
Moreover, for every $f\in L^1_{d\mu}\cap L^2_{d\mu}$, we have
\begin{equation}
\label{1-2}
	\|e^{tL_*}f\|_{L^2_{d\mu}}\leq Ct^{-\frac{N-\alpha}{2(2-\alpha)}}\|f\|_{L^1_{d\mu}}
\end{equation}
and
\begin{align}
\label{1-3}
	\|L_{*} e^{tL_{\ast}}f\|_{L^2_{d\mu}}
		\leq Ct^{-\frac{N-\alpha}{2(2-\alpha)}-1}\|f\|_{L^1_{d\mu}}.
\end{align}
\end{lemma}

\begin{remark}\label{rem2.3}
Applying the Riesz--Thorin theorem, we deduce from \eqref{1-2} and 
Lemma \ref{L*} that 
the following $L^2_{d\mu}$-$L^p_{d\mu}$ estimates with $1\leq p\leq 2$ also hold:
\begin{align}
\label{1-4}
	\|L_{*} e^{tL_{\ast}}f\|_{L^2_{d\mu}}
		\leq Ct^{-\frac{N-\alpha}{2(2-\alpha)}(\frac{1}{p}-\frac{1}{2})-1}\|f\|_{L^p_{d\mu}}, 
\quad f\in L^p_{d\mu}\cap L^2_{d\mu}.
\end{align}
\end{remark}

\section{Weighted energy estimates}

In this section we consider the weighted energy estimates for solutions of 
\eqref{dw}. 
First we construct them for compactly supported initial data. Then 
by the standard approximation argument we establish them for 
all reasonable initial data. 
The crucial point is to derive several estimates
which are uniform for the size of the support of initial data.

\subsection{Weighted energy estimates with compactly supported initial data}
For simplicity we will use 
\begin{gather*}
\mathcal{H}_c=\Big\{(f,g) \in (H^2\cap H^1_0(\Omega)) \times H^1_0(\Omega)\;;\;
\text{$f$ and $g$ are compactly supported in }\R^N
\Big\}.
\end{gather*}
Then the finite propagation property gives the following lemma.
\begin{lemma}\label{finite.speed}
Let $u$ be a solution of \eqref{dw} with $(u_0,u_1)\in \mathcal{H}_c.$ 
Then $(u(t),\pa_tu(t))\in \mathcal{H}_c$ for every $t\geq 0$.
\end{lemma}

\subsubsection{Estimates for $\nabla u$ and $\pa_t u$ with weight function $\Psi^\beta$}

Here we define the weighted energy functionals 
which are useful in the present paper. 

\begin{definition}\label{energy}
For $\beta\in \R$ and 
for the solution $w$ of \eqref{dw} with initial data 
$(f,g)\in (H^2\cap H^1_0(\Omega))\times H^1_0(\Omega)$, 
we define
\begin{align*}
E_{\pa x}^{\beta}[t_0,w](t)
&:=
\int_{\Omega}
   |\nabla w(x,t)|^2
   \Psi^{\beta}(x,t_0+t)
\,dx, 
\\
E_{\pa t}^{\beta}[t_0,w](t)
&:=
\int_{\Omega}
   |\pa_tw(x,t)|^2
   \Psi^{\beta}(x,t_0+t)
\,dx,
\\
E_{a}^{\beta}[t_0,w](t)
&:=
\int_{\Omega}
   |w(x,t)|^2a(x)
   \Psi^{\beta}(x,t_0+t)
\,dx.
\end{align*}
Note that these are finite in particular if $(f,g)\in \mathcal{H}_c$ (see Lemma \ref{finite.speed}).
\end{definition}

Throughout this paper, we will use the notation $\Psi^\beta=\Psi^\beta(x,t_0+t)$, for simplicity.

\begin{lemma}\label{energy1}
Let $u$ be a solution of \eqref{dw} with initial data 
$(u_0,u_1)\in \mathcal{H}_c$
and let $\beta \ge 0$.
Then there exist constants $t_1=t_1(\alpha,\beta)\geq 1$ 
and $K_1=K_1(\beta)>0$ such that 
if $t_0\geq t_1$, then 
\begin{align*}
\frac{d}{dt}\Big[E_{\pa x}^{\beta}[t_0,u](t)+E_{\pa t}^{\beta}[t_0,u](t)\Big]
&\leq 
   -\frac{1}{2}
E_a^{\beta}[t_0,\pa_t u](t)
+
K_1
E_{\pa x}^{\beta-1}[t_0,u](t), \quad t\geq 0.
\end{align*}
\end{lemma}

\begin{proof}
Firstly we easily see that 
\begin{align*}
\frac{d}{dt}E_{\pa t}^{\beta}[t_0,u](t)
=
2\int_{\Omega}
   \pa_tu(\pa_t^2u)
   \Psi^{\beta}
\,dx
+
\beta\int_{\Omega}
   |\pa_tu|^2
   \Psi^{\beta-1}
\,dx.
\end{align*}
Secondly by integration by parts we deduce
\begin{align*}
\frac{d}{dt}E_{\pa x}^{\beta}[t_0,u](t)
&=
2\int_{\Omega}
   \nabla (\pa_tu)\cdot \nabla u
   \Psi^{\beta}
\,dx
+
\beta
\int_{\Omega}
   |\nabla u|^2
   \Psi^{\beta-1}
\,dx
\\
&=
-2\int_{\Omega}
   (\pa_tu)\Delta u
   \Psi^{\beta}
\,dx
-2\beta\int_{\Omega}
   (\pa_tu)
   \nabla u\cdot\frac{x|x|^{-\alpha}}{2-\alpha}
   \Psi^{\beta-1}
\,dx
+
\beta
\int_{\Omega}
   |\nabla u|^2
   \Psi^{\beta-1}
\,dx
\end{align*}
Then the Schwarz inequality yields that
\begin{align*}
\frac{d}{dt}E_{\pa x}^{\beta}[t_0,u](t)
&\leq 
-2\int_{\Omega}
   (\pa_tu)\Delta u
   \Psi^{\beta}
\,dx
+
\int_{\Omega}
   |\pa_tu|^2|x|^{-\alpha}
   \Psi^{\beta}
\,dx
\\
&\quad+
\beta^2
\int_{\Omega}
   |\nabla u|^2\frac{|x|^{2-\alpha}}{(2-\alpha)^2}
   \Psi^{\beta-2}
\,dx
+
\beta
\int_{\Omega}
   |\nabla u|^2
   \Psi^{\beta-1}
\,dx
\\
&\leq 
-2\int_{\Omega}
   (\pa_tu)\Delta u
   \Psi^{\beta}
\,dx
+
\int_{\Omega}
   |\pa_tu|^2|x|^{-\alpha}
   \Psi^{\beta}
\,dx
+
\beta(1+\beta)
\int_{\Omega}
   |\nabla u|^2
   \Psi^{\beta-1}
\,dx.
\end{align*}
Therefore combining the above estimates with \eqref{dw} implies that 
\begin{align*}
&\frac{d}{dt}\Big[E_{\pa x}^{\beta}[t_0,u](t)
+E_{\pa t}^{\beta}[t_0,u](t)\Big]
\\
&=
2\int_{\Omega}
   (\pa_tu)(\pa_t^2u-\Delta u)
   \Psi^{\beta}
\,dx
+
\beta
\int_{\Omega}
   |\pa_tu|^2
   \Psi^{\beta-1}
\,dx
\\
&\quad
+
\int_{\Omega}
   |\pa_tu|^2|x|^{-\alpha}
   \Psi^{\beta}
\,dx
+
\beta(1+\beta)
\int_{\Omega}
   |\nabla u|^2
   \Psi^{\beta-1}
\,dx
\\
&\leq
-
\int_{\Omega}
   |\pa_tu|^2
   |x|^{-\alpha}\Psi^{\beta}
\,dx
+
\beta
\int_{\Omega}
   |\pa_tu|^2
   \Psi^{\beta-1}
\,dx
+
\beta(1+\beta)
\int_{\Omega}
   |\nabla u|^2
   \Psi^{\beta-1}
\,dx.
\end{align*}
Thus by noticing that $\alpha\geq 0$ and therefore
$
\Psi^{-\frac{\alpha}{2-\alpha}}
\leq 
(2-\alpha)^\frac{2\alpha}{2-\alpha}|x|^{-\alpha}
$
we obtain
\begin{align*}
\frac{d}{dt}\Big[E_{\pa x}^{\beta}[t_0,u](t)+E_{\pa t}^{\beta}[t_0,u](t)\Big]
&\leq 
\int_{\Omega}
   |\pa_tu|^2
   \left(
   -1
   +
   \beta
   (2-\alpha)^{\frac{2\alpha}{2-\alpha}}
  \Psi^{-\frac{2(1-\alpha)}{2-\alpha}}
   \right)
   |x|^{-\alpha}\Psi^{\beta}
\,dx
\\
&\quad +
\beta(1+\beta)
\int_{\Omega}
   |\nabla u|^2
   \Psi^{\beta-1}
\,dx.
\end{align*}
If 
$t_0\geq t_1:=
1+
(2\beta)^{\frac{2-\alpha}{2(1-\alpha)}}(2-\alpha)^\frac{\alpha}{1-\alpha}$, 
we obtain the desired inequality.
\end{proof}

\subsubsection{Weighted energy estimates for the case $\beta\leq 1$}

\begin{lemma}\label{energy2-pre}
Let $u$ be a solution of \eqref{dw} with initial data 
$(u_0,u_1)\in \mathcal{H}_c$. 
Then for every $t_0\geq 1$ and $t\geq 0$, 
\begin{align*}
\frac{d}{dt}\left[
2\int_{\Omega}u\pa_tu\,dx
+
E_{a}^0[t_0,u](t)
\right]
&
=
2
E_{\pa t}^0[t_0, u](t)
-
2E_{\pa x}^0[t_0,u](t).
\end{align*}
\end{lemma}
\begin{proof}
By \eqref{dw} we have
\begin{align*}
\frac{d}{dt}\left[
2\int_{\Omega}u\pa_tu\,dx
+
\int_{\Omega}
|x|^{-\alpha}|u|^2\,dx
\right]
&
=
2\int_{\Omega}(\pa_tu)^2\,dx
+
2\int_{\Omega}u\Big(\pa_t^2u+|x|^{-\alpha}\pa_tu\Big)\,dx
\\
&
=
2\int_{\Omega}(\pa_tu)^2\,dx
+
2\int_{\Omega}u\Delta u\,dx.
\end{align*}
Integration by parts implies the desired assertion. 
\end{proof}
\begin{proposition}\label{en.decay0}
If $\frac{\alpha}{2-\alpha}\leq \beta\leq 1$, then there exists a constant
$M_1=M_1(\alpha,\beta,N) > 0$ such that
for every $t_0\geq 1$ and $t\geq 0$, 
\begin{align*}
&
 E_{\pa x}^{\beta}[t_0,u](t)
+E_{\pa t}^{\beta}[t_0,u](t)
+E_{a}^0[t_0,u](t)
+
\int_0^t
E_a^{\beta}[t_0,\pa_t u](s)\,ds
\\
&\leq 
M_1
\left(
\int_{\Omega}
\Big(|\nabla u_0|^2+|u_1|^2\Big)
|x|^{(2-\alpha)\beta}
\,dx+
\int_\Omega|u_0|^2
|x|^{-\alpha}
\,dx
\right).
\end{align*}
In particular, if $\beta=1$, then one has
\begin{align*}
&
 E_{\pa x}^{1}[t_0,u](t)
+E_{\pa t}^{1}[t_0,u](t)
+E_{a}^0[t_0,u](t)
+
\int_0^t
E_a^{1}[t_0,\pa_t u](s)\,ds
\leq 
M_1
\int_{\Omega}
\Big(|\nabla u_0|^2+|u_1|^2\Big)
|x|^{2-\alpha}
\,dx.
\end{align*}
\end{proposition}
\begin{proof}
In the case $\frac{\alpha}{2-\alpha}<\beta\leq 1$,
by Lemmas \ref{energy1} and \ref{energy2-pre}, we have
\begin{align*}
&
\frac{d}{dt}\left[
E_{\pa x}^{\beta}[t_0,u](t)+E_{\pa t}^{\beta}[t_0,u](t)
+
\frac{K_1}{2}\left(
2\int_{\Omega}u\pa_tu\,dx
+
E_{a}^0[t_0,u](t)
\right)
\right]
\\
&\leq 
   -\frac{1}{2}
E_a^{\beta}[t_0,\pa_t u](t)
+
K_1
E_{\pa t}^0[t_0,\pa_t u](t)
+
K_1
E_{\pa x}^{\beta-1}[t_0,u](t)
-
K_1
E_{\pa x}^0[t_0,u](t)
\\
&\leq
\left( 
-\frac{1}{2}
+K_1(2-\alpha)^{\frac{2\alpha}{2-\alpha}}
t_0^{-\beta+\frac{\alpha}{2-\alpha}}
\right)
E_a^{\beta}[t_0,\pa_t u](t)
+
K_1
(t_0^{\beta-1}
-1)
E_{\pa x}^0[t_0,u](t).
\end{align*}
Taking $t_0'\geq t_1$ such that 
$4K_1(2-\alpha)^{\frac{2\alpha}{2-\alpha}}\leq t_0^{\beta-\frac{\alpha}{2-\alpha}}
$, we deduce
\[
\frac{d}{dt}\left[
E_{\pa x}^{\beta}[t_0,u](t)+E_{\pa t}^{\beta}[t_0,u](t)
+
\frac{K_1}{2}\left(
2\int_{\Omega}u\pa_tu\,dx
+
E_{a}^0[t_0,u](t)
\right)
\right]
\leq 
   -
   \frac{1}{4}
E_a^{\beta}[t_0,\pa_t u](t).
\]
Noting that 
\begin{align*}
\left|
2\int_{\Omega}u\pa_tu\,dx
\right|
&\leq 
\frac{1}{2}
E_{a}^0[t_0,u](t)
+
2
\int_{\Omega}
   |\pa_tu|^2|x|^{\alpha}
\,dx
\\
&\leq 
\frac{1}{2}
E_{a}^0[t_0,u](t)
+
2(2-\alpha)^{\frac{2\alpha}{2-\alpha}}t_0^{-\beta+\frac{\alpha}{2-\alpha}}
E_{\pa t}^{\beta}[t_0,u](t)
\\
&\leq 
\frac{1}{2}
E_{a}^0[t_0,u](t)
+
\frac{1}{2K_1}
E_{\pa t}^{\beta}[t_0,u](t),  
\end{align*}
we have
\begin{align*}
&E_{\pa x}^{\beta}[t_0,u](t)
+\frac{3}{4}E_{\pa t}^{\beta}[t_0,u](t)
+\frac{K_1}{4}
E_{a}^0[t_0,u](t)
+
\frac{1}{4}
\int_0^t
E_a^{\beta}[t_0,\pa_t u](s)\,ds
\\
&\leq 
E_{\pa x}^{\beta}[t_0,u](0)
+\frac{5}{4}E_{\pa t}^{\beta}[t_0,u](0)
+\frac{3K_1}{4}
E_{a}^0[t_0,u](0).
\end{align*}
If $\beta=1$, then a weighted Hardy inequality 
\[
\left(\frac{N-\alpha}{2}\right)^2
\int_{\Omega}|u|^2|x|^{-\alpha}\,dx
\leq \int_{\Omega}|\nabla u|^2|x|^{2-\alpha}\,dx
\]
implies the desired estimate (see e.g., 
Mitidieri \cite{Mitidieri00} and also
Arendt--Goldstein-Goldstein \cite{AGG06}).
On the other hand, if $\beta=\frac{\alpha}{2-\alpha}$, then 
\begin{align*}
&
\frac{d}{dt}\left[
E_{\pa x}^{\frac{\alpha}{2-\alpha}}[t_0,u](t)+E_{\pa t}^{\frac{\alpha}{2-\alpha}}[t_0,u](t)
+
\nu\left(
2\int_{\Omega}u\pa_tu\,dx
+
E_{a}^0[t_0,u](t)
\right)
\right]
\\
&\leq 
   -\frac{1}{2}
E_a^{\frac{\alpha}{2-\alpha}}[t_0,\pa_t u](t)
+
2\nu
E_{\pa t}^0[t_0,\pa_t u](t)
+
K_1
E_{\pa x}^{-\frac{2(1-\alpha)}{2-\alpha}}[t_0,u](t)
-
2\nu
E_{\pa x}^0[t_0,u](t)
\\
&\leq
\left( 
-\frac{1}{2}
+2\nu(2-\alpha)^{\frac{2\alpha}{2-\alpha}}
\right)
E_a^{\frac{\alpha}{2-\alpha}}[t_0,\pa_t u](t)
+
(K_1t_0^{-\frac{2(1-\alpha)}{2-\alpha}}
-2\nu)
E_{\pa x}^0[t_0,u](t).
\end{align*}
Therefore taking $\nu=8^{-1}(2-\alpha)^{-\frac{2\alpha}{2-\alpha}}$ 
and $t_0\geq t_1$ such that 
$t_0^{\frac{2(1-\alpha)}{2-\alpha}}\geq 4K_1(2-\alpha)^{\frac{2\alpha}{2-\alpha}}$, 
we obtain the same inequality as the previous case. 
\end{proof}
\subsubsection{Estimates for $u$ with weight function $\Phi_\lambda$}

Here we define new weighted energy functional
via Kummer's confluent geometric functions,
which plays a crucial role in this paper.  
\begin{definition}
For $\lambda\in [0,\frac{N-\alpha}{2-\alpha})$, 
choose $\ep_*\in(0,\frac{1}{3}]$ such that 
$\lambda=(1-3\ep_*)\frac{N-\alpha}{2-\alpha}$.
Define 
\[
\lambda_*=\lambda(1-2\ep_*)^{-1}<\frac{N-\alpha}{2-\alpha}
\]
and 
\begin{align*}
E_{\Phi}^\lambda[t_0,u](t)
&=
\int_{\Omega}
   \left(2u(x,t)\pa_tu(x,t)+a(x)|u|^2\right)
   [\Phi_{\lambda_*}(x,t_0+t)]^{-1+2\ep_*}
\,dx.
\end{align*}
\end{definition}

\begin{lemma}\label{hardy1}
For every $w\in H^1_0(\Omega)$ 
having a compact support on $\R^N$ and $\lambda>-\frac{N-2}{2-\alpha}$
\begin{equation}\label{hardy.eq}
\int_{\Omega}
   |w|^2|x|^{-\alpha}\Psi^{\lambda-1}
\,dx
\leq 
4
\min\left\{\frac{N-\alpha}{2-\alpha},\frac{N-2}{2-\alpha}+\lambda\right\}^{-2}
\int_{\Omega}
   |\nabla w|^2\Psi^{\lambda}
\,dx.
\end{equation}
\end{lemma}

\begin{proof}
Observe that
\begin{align*}
{\rm div}\left(\Psi^{\lambda-1}\nabla \Psi\right)
&=
\Delta\Psi\Psi^{\lambda-1}
+(\lambda-1)
\Psi^{\lambda-2}|\nabla \Psi|^2
\\
&=
\left[\frac{N-\alpha}{2-\alpha} (t_0+t)
+
\left(\frac{N-\alpha}{2-\alpha}+\lambda-1\right)\frac{|x|^{2-\alpha}}{(2-\alpha)^2}
\right]|x|^{-\alpha}\Psi^{\lambda-2}
\\
&\geq 
\min\left\{\frac{N-\alpha}{2-\alpha},\frac{N-2}{2-\alpha}+\lambda\right\}
|x|^{-\alpha}\Psi^{\lambda-1}.
\end{align*}
On the other hand, integration by parts and Schwarz's inequality yield that 
\begin{align*}
\int_{\Omega}
   |w|^2{\rm div}\left(\Psi^{\lambda-1}\nabla \Psi\right)
\,dx
&=
2
\int_{\Omega}
   w\nabla w \cdot \nabla \Psi\Psi^{\lambda-1}
\,dx
\\
&\leq 
2
\left(
\int_{\Omega}
   |w|^2|x|^{-\alpha}\Psi^{\lambda-1}
\,dx
\right)^\frac{1}{2}
\left(
\int_{\Omega}
   |\nabla w|^2|x|^{\alpha}|\nabla \Psi|^2
   \Psi^{\lambda-1}
\,dx
\right)^\frac{1}{2}
\\
&\leq 
2
\left(
\int_{\Omega}
   |w|^2|x|^{-\alpha}\Psi^{\lambda-1}
\,dx
\right)^\frac{1}{2}
\left(
\int_{\Omega}
   |\nabla w|^2\Psi^{\lambda}
\,dx
\right)^\frac{1}{2}.
\end{align*}
Combining the estimates above, we obtain \eqref{hardy.eq}.
\end{proof}

\begin{lemma}\label{hardy2}
For $w\in H^1_0(\Omega)$ having a compact support in $\R^N$, one has
\begin{align}\label{aux}
(1-\ep_*)
\int_{\Omega}
   |w|^2|\nabla\Phi_{\lambda_*}|^2\Phi_{\lambda_*}^{-3+2\ep_*}
\,dx
\leq 
\frac{1}{1-\ep_*}\int_{\Omega}
   |\nabla w|^2
   \Phi_{\lambda_*}^{-1+2\ep_*}
\,dx
+
\int_{\Omega}
   |w|^2\Delta\Phi_{\lambda_*}\Phi_{\lambda_*}^{-2+2\ep_*}
\,dx.
\end{align}
\end{lemma}
\begin{proof}
Putting $v=\Phi_{\lambda_*}^{-1+\ep_*}w$ (that is, $w=\Phi_{\lambda_*}^{1-\ep_*}v$) 
and using integration by parts,
we have 
\begin{align*}
\int_{\Omega}
   |\nabla w|^2
   \Phi_{\lambda_*}^{-1+2\ep_*}
\,dx
&=
\int_{\Omega}
   |\nabla (\Phi_{\lambda_*}^{1-\ep_*}v)|^2
   \Phi_{\lambda}^{-1+2\ep_*}
\,dx
\\
&\geq 
(1-\ep_*)
\left[
2\int_{\Omega}
   v\nabla v\cdot\nabla\Phi_{\lambda_*}
\,dx
+
(1-\ep_*)
\int_{\Omega}
   v^2|\nabla\Phi_{\lambda_*}|^2\Phi_{\lambda_*}^{-1}
\,dx
\right]
\\
&=
(1-\ep_*)
\left[
-\int_{\Omega}
   v^2\Delta\Phi_{\lambda_*}
\,dx
+
(1-\ep_*)
\int_{\Omega}
   v^2|\nabla\Phi_{\lambda_*}|^2\Phi_{\lambda_*}^{-1}
\,dx
\right].
\end{align*}
Rewriting $v$ in terms of $w$, we deduce \eqref{aux}.
\end{proof}

\begin{lemma}\label{energy2}
There exist constants $t_2=t_2(\lambda,\alpha,N)$, 
$\eta=\eta(\lambda)$ and $K_2=K_2(\lambda,\alpha,N)>0$ 
such that for every $t_0\geq t_2$, 
\begin{align*}
\frac{d}{dt}\Big[
E_\Phi^{\lambda}[t_0,u](t)
\Big]
&
\leq 
-\eta
E_{\pa x}^\lambda[t_0,u](t)
+
K_2 E_{a}^{\lambda+\frac{\alpha}{2-\alpha}}[t_0,\pa_t u](t), 
\quad t\geq 0.
\end{align*}
\end{lemma}
\begin{proof}
We see from Lemma \ref{Phi.beta.lem} {\bf (ii)} and {\bf (i)} that 
\begin{align*}
&\frac{d}{dt}
\int_{\Omega}
   2u(\pa_tu)
   \Phi_{\lambda_*}^{-1+2\ep_*}
\,dx
\\
&=
2\int_{\Omega}
   u(\pa_t^2u)
   \Phi_{\lambda_*}^{-1+2\ep_*}
\,dx
+
2\int_{\Omega}
   |\pa_t u|^2
   \Phi_{\lambda_*}^{-1+2\ep_*}
\,dx
-
(1-2\ep_*)
\int_{\Omega}
   u(\pa_t u)
   \left(\pa_t\Phi_{\lambda_*}\right)
   \Phi_{\lambda_*}^{-2+2\ep_*}
\,dx
\\
&=
2\int_{\Omega}
   u(\pa_t^2 u)
   \Phi_{\lambda_*}^{-1+2\ep_*}
\,dx
+
2\int_{\Omega}
   |\pa_t u|^2
   \Phi_{\lambda_*}^{-1+2\ep_*}
\,dx
+(1-2\ep_*)\lambda_*
\int_{\Omega}
   u(\pa_t u)
   \Phi_{\lambda_*+1}
   \Phi_{\lambda_*}^{-2+2\ep_*}
\,dx
\end{align*}
and 
\begin{align*}
\frac{d}{dt}
\int_{\Omega}
   |u|^2|x|^{-\alpha}
   \Phi_{\lambda_*}^{-1+2\ep_*}
\,dx
&=
2\int_{\Omega}
   u(\pa_tu)|x|^{-\alpha}
   \Phi_{\lambda_*}^{-1+2\ep_*}
\,dx
-
(1-2\ep_*)
\int_{\Omega}
   |u|^2|x|^{-\alpha}
   \left(\pa_t\Phi_{\lambda_*}\right)
   \Phi_{\lambda_*}^{-2+2\ep_*}
\,dx
\\
&=
2\int_{\Omega}
   u(\pa_t u)|x|^{-\alpha}
   \Phi_{\lambda_*}^{-1+2\ep_*}
\,dx
-
(1-2\ep_*)
\int_{\Omega}
   |u|^2
   \left(\Delta \Phi_{\lambda_*}\right)
   \Phi_{\lambda_*}^{-2+2\ep_*}
\,dx.
\end{align*}
Combining the equalities above we have 
\begin{align*}
\frac{d}{dt}\Big[
E_\Phi^\lambda[t_0,u](t)
\Big]
&=
2\int_{\Omega}
   u\Big(\pa_t^2 u+|x|^{-\alpha}\pa_tu\Big)
   \Phi_{\lambda_*}^{-1+2\ep_*}
\,dx
+
2\int_{\Omega}
   |\pa_tu|^2
   \Phi_{\lambda_*}^{-1+2\ep_*}
\,dx
\\
&\quad
-
(1-2\ep_*)
\int_{\Omega}
   |u|^2
   \left(\Delta \Phi_{\lambda_*}\right)
   \Phi_{\lambda_*}^{-2+2\ep_*}
\,dx
+
\lambda
\int_{\Omega}
   u(\pa_t u)
   \Phi_{\lambda_*+1}\Phi_{\lambda_*}^{-2+2\ep_*}
\,dx.
\end{align*}
On the other hand, we see from integration by parts 
and Lemma \ref{hardy2} that
\begin{align*}
&2\int_{\Omega}
   u(\Delta u)
   \Phi_{\lambda_*}^{-1+2\ep_*}
\,dx
\\
&=
2(1-2\ep_*)\int_{\Omega}
   u(\nabla u\cdot\nabla\Phi_{\lambda_*}) 
   \Phi_{\lambda_*}^{-2+2\ep_*}
\,dx
-2\int_{\Omega}
   |\nabla u|^2
   \Phi_{\lambda_*}^{-1+2\ep_*}
\,dx
\\
&=
(1-2\ep_*)
\left[
-\int_{\Omega}
   |u|^2(\Delta\Phi_{\lambda_*}) 
   \Phi_{\lambda_*}^{-2+2\ep_*}
\,dx
+
2(1-\ep_*)\int_{\Omega}
   |u|^2|\nabla \Phi_{\lambda_*}|^2 
   \Phi_{\lambda_*}^{-3+2\ep_*}
\,dx
\right]
-2
\int_{\Omega}
   |\nabla u|^2
   \Phi_{\lambda_*}^{-1+2\ep_*}
\,dx
\\
&\leq
(1-2\ep_*)
\int_{\Omega}
   |u|^2(\Delta\Phi_{\lambda_*}) 
   \Phi_{\lambda_*}^{-2+2\ep_*}
\,dx
-\frac{2\ep_*}{1-\ep_*}
\int_{\Omega}
   |\nabla u|^2
   \Phi_{\lambda_*}^{-1+2\ep_*}
\,dx.
\end{align*}
Noting that $\pa_t^2u+|x|^{-\alpha}\pa_tu=\Delta u$, 
using above two estimates, we have
\begin{align*}
\frac{d}{dt}\Big[
E_\Phi^\lambda[t_0,u](t)
\Big]
&\leq 
-\frac{2\ep_*}{1-\ep_*}
\int_{\Omega}
   |\nabla u|^2
   \Phi_{\lambda_*}^{-1+2\ep_*}
\,dx
+
2\int_{\Omega}
   |\pa_tu|^2
   \Phi_{\lambda_*}^{-1+2\ep_*}
\,dx
\\
&\quad
+
\lambda
\int_{\Omega}
   u(\pa_t u)
   \Phi_{\lambda_*+1}\Phi_{\lambda_*}^{-2+2\ep_*}
\,dx.
\end{align*}
Then the assertions {\bf (iii)} and {\bf (iv)} in Lemma \ref{Phi.beta.lem}  
imply that 
\begin{align*}
\frac{d}{dt}\Big[
E_\Phi^\lambda[t_0,u](t)
\Big]
&\leq 
-\frac{2\ep_*}{(1-\ep_*)C_{\lambda_*}^{1-2\ep_*}}
\int_{\Omega}
   |\nabla u|^2
   \Psi^{\lambda}
\,dx
+
\frac{2}{c_{\lambda_*}^{1-2\ep_*}}\int_{\Omega}
   |\pa_tu|^2
   \Psi^{\lambda}
\,dx
\\
&\quad
+
\frac{\lambda C_{\lambda_*+1}}{c_{\lambda_*}^{2-2\ep_*}}
\int_{\Omega}
   |u|\,|\pa_t u|
   \Psi^{\lambda-1}
\,dx.
\end{align*}
Finally, Lemma \ref{hardy1} implies 
\begin{align*}
\frac{d}{dt}\Big[
E_\Phi^\lambda[t_0,u](t)
\Big]
&\leq 
-
\eta\left(
\int_{\Omega}
   |\nabla u|^2
   \Psi^{\lambda}
\,dx
+
\int_{\Omega}
   |u|^2|x|^{-\alpha}
   \Psi^{\lambda-1}
\,dx
\right)
\\
&\quad
+K_2\left(\int_{\Omega}
   |\pa_tu|^2
   \Psi^{\lambda}
\,dx
+
\int_{\Omega}
   |u|^2
   \Psi^{\lambda-2}
\,dx
\right),
\end{align*}
with the constants $\eta=\eta(\lambda)>0$, $K_2=K_2(\alpha,\lambda,N)>0$.
Observe that 
\[
\Psi^{-1}
=
\left(t_0+t+\frac{|x|^{2-\alpha}}{(2-\alpha)^2}\right)^{-\frac{2(1-\alpha)}{2-\alpha}}
\left(t_0+t+\frac{|x|^{2-\alpha}}{(2-\alpha)^2}\right)^{-\frac{\alpha}{2-\alpha}}
\leq 
t_0^{-\frac{2(1-\alpha)}{2-\alpha}}
(2-\alpha)^{\frac{2\alpha}{2-\alpha}}
|x|^{-\alpha}.
\]
By choosing $t_0\geq t_2$ with $t_2$ sufficiently large, 
we obtain the desired inequality.
\end{proof}

\subsubsection{Weighted energy estimates for the case
$1<\beta<\frac{N-\alpha}{2-\alpha}+1$}

\begin{lemma}\label{lower.bdd}
If $\lambda\leq \beta-\frac{\alpha}{2-\alpha}$, 
then for every $\delta>0$, 
there exists a constant $K_3=K_3(\alpha,\lambda,\delta)>0$ 
such that for every $t_0\geq 1$ and $t\geq 0$, 
\begin{align}\label{cross}
\left|
\int_{\Omega}
   2u\pa_tu
   \Phi_{\lambda_*}^{-1+2\ep}
\,dx
\right|
\leq 
K_3t_0^{-(\beta-\lambda-\frac{\alpha}{2-\alpha})}E_{\pa t}^{\beta}[t_0,u](t)
+
\delta 
E_{a}^{\lambda}[t_0,u](t).
\end{align}
In particular, 
\[
-
K_3t_0^{-(\beta-\lambda-\frac{\alpha}{2-\alpha})}E_{\pa t}^{\beta}[t_0,u](t)
+\delta_0 E_{a}^{\lambda}[t_0,u](t)
\leq
E_{\Phi}^\lambda[t_0,u](t)
\leq 
K_3t_0^{-(\beta-\lambda-\frac{\alpha}{2-\alpha})}E_{\pa t}^{\beta}[t_0,u](t)
+2 E_{a}^{\lambda}[t_0,u](t). 
\]
\end{lemma}

\begin{proof}
Young's inequality yields that
\begin{align*}
\left|
\int_{\Omega}
   2u\pa_tu
   \Phi_{\lambda_*}^{-1+2\ep_*}
\,dx
\right|
&\leq 
2K'
\int_{\Omega}
   |u|\,|\pa_tu|
   \Psi^\lambda
\,dx
\\
&\leq 
\frac{(K')^2}{\delta}
\int_{\Omega}
   |\pa_t u|^2
   |x|^{\alpha}   
   \Psi^{\lambda}
\,dx
+
\delta
\int_{\Omega}
   |u|^2
   |x|^{-\alpha}
   \Psi^{\lambda}
\,dx
\\
&\leq 
\frac{(K')^2}{\delta}(2-\alpha)^{\frac{2\alpha}{2-\alpha}}
\int_{\Omega}
   |\pa_t u|^2
   \Psi^{\lambda+\frac{\alpha}{2-\alpha}}
\,dx
+
\delta
\int_{\Omega}
   |u|^2
   |x|^{-\alpha}
   \Psi^{\lambda}
\,dx
\\
&\leq 
\frac{(K')^2}{\delta}(2-\alpha)^{\frac{2\alpha}{2-\alpha}}
t_0^{\lambda+\frac{\alpha}{2-\alpha}-\beta}
\int_{\Omega}
   |\pa_t u|^2
   \Psi^{\beta}
\,dx
+
\delta
\int_{\Omega}
   |u|^2
   |x|^{-\alpha}
   \Psi^{\lambda}
\,dx.
\end{align*}
This implies both of inequalities in the second assertion.
\end{proof}

\begin{proposition}\label{en.decay1}
Let $u$ be a solution of \eqref{dw} with initial data
$(u_0,u_1)\in \mathcal{H}_c$. 
Assume $\beta\in (1,\frac{N-\alpha}{2-\alpha}+1)$.
Then there exists $M_1=M_1(\alpha,\beta,N)$ such that 
\begin{align}\label{main.est1}
   \Big(
   E_{\pa x}^{\beta}[t_0,u](t)
   +
   E_{\pa t}^{\beta}[t_0,u](t)
   \Big)
   + 
\int_{0}^t
E_a^{\beta}[t_0,\pa_t u](s)
\,ds\leq 
M_1
\int_{\Omega}
   \Big(|\nabla u_0|^2+|u_1|^2\Big)|x|^{(2-\alpha)\beta}
\,dx.
\end{align}
\end{proposition}

\begin{remark}
From Proposition \ref{en.decay1} we can derive 
the usual energy decay estimate
\[
\int_{\Omega}|\nabla u(x,t)|^2\,dx
+
\int_{\Omega}|\pa_tu(x,t)|^2\,dx
\leq 
C(1+t)^{-\beta}
\]
with the decay late $\beta<\frac{N-\alpha}{2-\alpha}+1$ 
which upper bound is nothing but the optimal constant obtained in \cite{ToYo09} and also \cite{So_Wa1}.
\end{remark}

\begin{proof}
We recall that Lemma \ref{energy1} and 
Lemma \ref{energy2} with 
$\lambda=\beta-1<\frac{N-\alpha}{2-\alpha}$ 
assert that
\begin{gather}
\frac{d}{dt}\Big[E_{\pa x}^{\beta}[t_0,u](t)+E_{\pa t}^{\beta}[t_0,u](t)\Big]
\leq 
   -\frac{1}{2}
E_a^{\beta}[t_0,\pa_t u](t)
+
K_1E_{\pa x}^{\beta-1}[t_0,u](t), 
\\
\frac{d}{dt}\Big[
E_\Phi^{\beta-1}[t_0,u](t)
\Big]
\leq 
-\eta
E_{\pa x}^{\beta-1}[t_0,u](t)
+
K_2 E_{a}^{\beta-1+\frac{\alpha}{2-\alpha}}[t_0,\pa_t u](t).
\end{gather}
Therefore we have
\begin{align*}
\frac{d}{dt}
\Big[
E_{\pa x}^{\beta}[t_0,u](t)
+E_{\pa t}^{\beta}[t_0,u](t)
+\frac{K_1}{\eta}
E_\Phi^{\beta-1}[t_0,u](t)
\Big]
&\leq 
   -\frac{1}{2}
E_a^{\beta}[t_0,\pa_t u](t)
+
\frac{K_1K_2}{\eta} E_{a}^{\beta-\frac{2(1-\alpha)}{2-\alpha}}[t_0,\pa_t u](t)
\\
&\leq 
   -\frac{1}{2}
E_a^{\beta}[t_0,\pa_t u](t)
+
\frac{K_1K_2}{\eta}t_0^{-\frac{2(1-\alpha)}{2-\alpha}} E_{a}^{\beta}[t_0,\pa_t u](t).
\end{align*}
Observe that 
\begin{align*}
   &E_{\pa x}^{\beta}[t_0,u](t)
   +E_{\pa t}^{\beta}[t_0,u](t)
   +\frac{K_1}{\eta}
E_\Phi^{\beta-1}[t_0,u](t)
\\
&\geq 
   E_{\pa x}^{\beta}[t_0,u](t)
   +\left(1-\frac{K_1K_3}{\eta}t_0^{-\frac{2(1-\alpha)}{2-\alpha}}\right)
   E_{\pa t}^{\beta}[t_0,u](t)
   + \frac{K_1\delta_0}{\eta}E_a^{\beta-1}[t_0,u](t)
\end{align*}
and
\begin{align*}
&   E_{\pa x}^{\beta}[t_0,u](0)
   +E_{\pa t}^{\beta}[t_0,u](0)
   +\nu E_\Phi^{\beta-1}[t_0,u](0)
\\
&\leq 
   E_{\pa x}^{\beta}[t_0,u](0)
   +\left(1+\frac{K_1K_3}{\eta}t_0^{-\frac{2(1-\alpha)}{2-\alpha}}\right)
E_{\pa t}^{\beta}[t_0,u](0)
   +\frac{2K_1}{\eta} E_a^{\beta-1}[t_0,u](0)
\\
&\leq 
   \left(1+
   \frac{8K_1}{\eta}
   \min\left\{\frac{N-\alpha}{2-\alpha},\frac{N-2}{2-\alpha}+\lambda\right\}^{-2}\right)
   E_{\pa x}^{\beta}[t_0,u](0)
   +\left(1+\frac{K_1K_3}{\eta}t_0^{-\frac{2(1-\alpha)}{2-\alpha}}\right)
E_{\pa t}^{\beta}[t_0,u](0).
\end{align*}
In this case by choosing $t_3\geq t_2$ 
such that 
$t_3^{-\frac{2(1-\alpha)}{2-\alpha}}\leq \min\{\frac{\eta}{4K_1K_2},\frac{\eta}{2K_1K_3}\}$,
 we obtain for $t_0\geq t_3$, 
\begin{align*}
\frac{d}{dt}
\Big[
   E_{\pa x}^{\beta}[t_0,u](t)
   +E_{\pa t}^{\beta}[t_0,u](t)
   +\frac{K_1}{\eta}
 E_\Phi^{\beta-1}[t_0,u](t)
\Big]
\leq 
   -\frac{1}{4}
E_a^{\beta}[t_0,\pa_t u](t).
\end{align*}
Proceeding the same argument as in the previous case, we obtain \eqref{main.est1}.
\end{proof}

\subsection{Weighted energy estimates for decaying initial data}

\begin{proof}[Proof of Theorem \ref{main1}]
Let $(u_0,u_1)$ satisfy \eqref{ass.thm}. 
Fix $\eta\in C_c^\infty(\R^N,[0,1])$ satisfying 
$\eta(x)=1$ for $x\in B(0,1)$ 
and 
$\eta(x)=0$ for $x\in \R^N\setminus B(0,2)$. Set 
for each $n\in \N$, 
\[
u_{0n}(x):=\eta(n^{-1}x)u_0(x), 
\quad 
u_{1n}(x):=\eta(n^{-1}x)u_1(x).
\] 
Then clearly we have $(u_{0n},u_{1n})\in \mathcal{H}_c$ for every $n\in\N$. 
Let $u_n$ be a solution of \eqref{dw} with initial data $(u_{0n},u_{1n})$.
Moreover, noting 
$|u_{1n}(x)|\leq |u_1(x)|$ and  
\begin{align*}
|\nabla u_{0n}(x)|
&\leq |\nabla u_0(x)|+n^{-1}|\nabla \eta(n^{-1}x)||u_0(x)|
\\
&\leq |\nabla u_0(x)|+|x|^{-1}\Big(\sup_{y\in\R^N}|y\nabla\eta(y)|\Big)|u_0(x)|,
\end{align*}
by Lemma \ref{hardy1}, we have
\begin{align*}
\int_{\R^N}
   |\nabla u_{0n}|^2|x|^{\gamma}
\,dx
&\leq 
2
\left(\int_{\R^N}
   |\nabla u_{0}|^2|x|^{\gamma}
\,dx
+
\int_{\R^N}
   |u_{0}|^2|x|^{\gamma-2}
\,dx
\right)
\\
&\leq 
2
\left(
1+
\frac{4}{(N+\gamma-2)^2}
\right)
\int_{\R^N}
   |\nabla u_{0}|^2|x|^{\gamma}
\,dx.
\end{align*}
On the other hand, we can check that $(u_{0n},u_{1n})\to (u_0,u_1)$ 
in $H^1_0(\Omega)\times L^2(\Omega)$ as $n\to \infty$. 
This means that the strong continuity of 
the semigroup $(u_0,u_1)\mapsto (u,\pa_tu)$ in $H^1_0(\Omega)\times L^2(\Omega)$ 
implies 
\[
(u_n,\pa_t u_n)\to (u,\pa_tu)
\text{ in }H^1_0(\Omega)\times L^2(\Omega)\text{ as }n\to \infty.
\]
Consequently, applying 
Proposition \ref{en.decay0} ($\gamma\leq 2-\alpha$)
or Proposition \ref{en.decay1} ($\gamma> 2-\alpha$) 
with $\beta=\frac{\gamma}{2-\alpha}$ 
and letting $n\to \infty$, 
we obtain the desired wighted energy estimates 
\eqref{w.est}.
\end{proof}
\section{Weight energy estimates for higher order derivatives}

We prove weighted energy estimates for higher order derivatives.
To state the assertion, we use the compatibility condition of 
order $k$, which is defined as follows:
The initial data $(u_0,u_1)$ satisfy the compatibility condition of 
order $k$ if $u_{j+1}=-\Delta u_{j-1}+|x|^{-\alpha}u_{j}\in L^2(\Omega)$ $(j=1,\ldots,k)$ can be successively defined with $(u_{j-1},u_{j})\in (H^2\cap H^1_0(\Omega))\times H^1_0(\Omega)$ for $j=1,\ldots,k$.

\begin{theorem}\label{main2}
Assume that $(u_0,u_1)$ satisfies the compatibility condition 
of order $k$ greater than $1$. 
Put $u_\ell=-\Delta u_{\ell-2}+a(x)u_{\ell-1}$ $(\ell=2,\ldots,k+1)$.
Suppose that 
\[
\int_{\Omega}\Big(|\nabla u_\ell|^{2}+|u_{\ell+1}|^{2}\Big)|x|^{\gamma+2\ell}\,dx
<\infty
\]
for all $\ell=0,\ldots, k$.
Then there exists $M_\ell=M_\ell(\alpha,\gamma,N,\ell)$ such that 
\begin{align}\label{main.estk}
   \Big(
   E_{\pa x}^{\frac{\gamma+2\ell}{2-\alpha}}[t_0,\pa_t^{\ell}u](t)
   +
   E_{\pa t}^{\frac{\gamma+2\ell}{2-\alpha}}[t_0,\pa_t^{\ell}u](t)
   \Big)
   + 
\int_{0}^t
E_a^{\frac{\gamma+2\ell}{2-\alpha}}[t_0,\pa_t^{\ell+1}u](s) 
\,ds\leq 
M_\ell\sum_{j=0}^\ell
\int_{\Omega}\Big(|\nabla u_j|^{2}+|u_{j+1}|^{2}\Big)|x|^{\gamma+2j}\,dx
\end{align}
\end{theorem}

\begin{lemma}\label{admit}
Under the assumption in Theorem \ref{main2}, one has
for $\ell=0,\ldots,k$, 
\begin{align*}
E_{\pa x}^{\frac{\gamma+2\ell}{2-\alpha}}[t_0,\pa_t^\ell u](0)
+
E_{\pa t}^{\frac{\gamma+2\ell}{2-\alpha}}[t_0,\pa_t^\ell u](0)
+
E_{a}^{\frac{\gamma+2\ell}{2-\alpha}-1}[t_0,\pa_t^\ell u](0)<\infty.
\end{align*}
\end{lemma}

Since the same approximation argument as in the proof of 
Theorem \ref{main1}
is also applicable 
in this case, we focus only on the case where the initial 
data $(u_0,u_1)$ have compact supports.

Here we introduce another energy functional $E_{*}^\lambda$ in stead of $E_{\Phi}^\lambda$.
\begin{definition}
For $\beta\in \R$ and 
for the solution $w$ of \eqref{dw} with initial data 
$(f,g)\in \mathcal{H}_c$, 
\[
E_*^{\lambda}[t_0,w](t):=
2\int_{\Omega}
w(x,t)\pa_tw(x,t)\Psi^{\lambda}(x,t_0+t)
\,dx.
\]
\end{definition}
\begin{remark}
$E_\Phi^\lambda$ is meaningful if $\lambda<\frac{N-\alpha}{2-\alpha}$, that is, 
$\Phi_{\lambda}$ is positive for all $x\in \R^N$. 
\end{remark}
\begin{lemma}\label{energy1-2}
Under the assumption in Theorem \ref{main2}, one has
\begin{align}
\nonumber
\frac{d}{dt}
\Big[
E_*^{\lambda}[t_0,u](t)
+
E_a^{\lambda}[t_0,u](t)
\Big]
&\leq 
-
E_{\pa x}^\lambda[t_0,u](t)
+
K_4(\beta)
E_{\pa t}^\lambda[t_0,u](t)
\\
\label{en1-2}
&\quad+
K_5(\alpha,\lambda)
E_{a}^{\lambda-1}[t_0,u](t)
\end{align}
with $K_4(\lambda)=2+\lambda$ and 
$K_5(\alpha,\lambda)
=\lambda\left(\lambda+(2-\alpha)^{\frac{2\alpha}{2-\alpha}}+1\right)$.
\end{lemma}

\begin{proof}
By a simple calculation we have
\begin{align*}
\frac{d}{dt}
\int_{\Omega}
\left(2u\pa_tu+|x|^{-\alpha}|u|^2\right)\Psi^{\lambda}\,dx
&=
2\int_{\Omega}
|\pa_tu|^2\Psi^{\lambda}\,dx
+
2\int_{\Omega}
u\left(\pa_{t}^2u+|x|^{-\alpha}u \right)\Psi^{\lambda}\,dx
\\
&\quad
+\lambda
\int_{\Omega}
\left(2u\pa_tu+|x|^{-\alpha}|u|^2\right)\Psi^{\lambda-1}\,dx.
\end{align*}
By using the equation in \eqref{dw} 
we see from 
 integration by parts twice 
that
\begin{align*}
2\int_{\Omega}
   u\left(\pa_{t}^2u+|x|^{-\alpha}u \right)\Psi^{\lambda}
\,dx
&=
2\int_{\Omega}
   u(\Delta u)\Psi^{\lambda}
\,dx
\\
&=
-2
\int_{\Omega}
   |\nabla u|^2\Psi^{\lambda}
\,dx
-
2\lambda
\int_{\Omega}
   u\nabla u\cdot\frac{|x|^{-\alpha}x}{2-\alpha}\Psi^{\lambda-1}
\,dx
\\
&\leq
-
\int_{\Omega}
   |\nabla u|^2\Psi^{\lambda}
\,dx
+
\lambda^2 
\int_{\Omega}
   |u|^2\frac{|x|^{2-2\alpha}}{(2-\alpha)^2}\Psi^{\lambda-2}
\,dx
\\
&\leq
-
\int_{\Omega}
   |\nabla u|^2\Psi^{\lambda}
\,dx
+
\lambda^2 
\int_{\Omega}
   |u|^2|x|^{-\alpha}\Psi^{\lambda-1}
\,dx.
\end{align*}
On the other hand, noting that $\Psi^{-\frac{\alpha}{2-\alpha}}\leq (2-\alpha)^{\frac{2\alpha}{2-\alpha}}|x|^{-\alpha}$, we have 
\begin{align*}
\lambda
\int_{\Omega}
\left(2u\pa_tu+|x|^{-\alpha}|u|^2\right)\Psi^{\lambda-1}
\,dx
&\leq 
\lambda
\int_{\Omega}
|\pa_tu|^2\Psi^{\lambda}
\,dx
+
\lambda\int_{\Omega}
|u|^2\Psi^{\lambda-2}
\,dx
+
\lambda\int_{\Omega}
|u|^2|x|^{-\alpha}\Psi^{\lambda-1}
\,dx
\\
&\leq 
\lambda
\int_{\Omega}
|\pa_tu|^2\Psi^{\lambda}
\,dx
+
\lambda\left((2-\alpha)^{\frac{2\alpha}{2-\alpha}}+1\right)
\int_{\Omega}
|u|^2|x|^{-\alpha}\Psi^{\lambda-1}
\,dx.
\end{align*}
And hence we have \eqref{en1-2}.
\end{proof}

\begin{lemma}\label{energy1-3}
Under the assumption in Theorem \ref{main2}, one has
\begin{align*}
&\frac{d}{dt}
\Big[
E_{\pa x}^{\beta}[t_0,u](t)
+
E_{\pa t}^{\beta}[t_0,u](t)
+
K_1\Big(
E_{*}^{\beta-\frac{\alpha}{2-\alpha}}[t_0,u](t)
+
E_{a}^{\beta-\frac{\alpha}{2-\alpha}}[t_0,u](t)
\Big)
\Big]
\\
&\leq
-\frac{1}{4}
E_{a}^{\beta}[t_0,\pa_t u](t)
+
K_6(\alpha,\beta)
E_{a}^{\beta-\frac{2}{2-\alpha}}[t_0,u](t).
\end{align*}
\end{lemma}

\begin{remark}
If $(u_0,u_1)$ has a compact support, then we can use the following estimate
\begin{align*}
&\frac{d}{dt}
\Big[
E_{\pa x}^{\beta}[t_0,u](t)
+
E_{\pa t}^{\beta}[t_0,u](t)
+
K_1\Big(
E_{*}^{\beta-1}[t_0,u](t)
+
E_{a}^{\beta-1}[t_0,u](t)
\Big)
\Big]
\\
&\leq
-\frac{1}{4}
E_{a}^{\beta}[t_0,\pa_t u](t)
+
\widetilde{K}_6(\alpha,\beta)
E_{a}^{\beta-2}[t_0,u](t).
\end{align*}
\end{remark}
\begin{proof}[Proof of Lemma \ref{energy1-3}]
Let $\nu>0$ be determined later. 
By using Lemmas \ref{energy1} and \ref{energy1-2},
we see that
\begin{align*}
&\frac{d}{dt}
\Big[
E_{\pa x}^{\beta}[t_0,u](t)
+
E_{\pa t}^{\beta}[t_0,u](t)
+
\nu\Big(
E_{*}^{\beta-\frac{\alpha}{2-\alpha}}[t_0,u](t)
+
E_{a}^{\beta-\frac{\alpha}{2-\alpha}}[t_0,u](t)
\Big)
\Big]
\\
&\leq-
\frac{1}{2}
E_{a}^{\beta}[t_0,\pa_t u](t)
+
K_1E_{\pa x}^{\beta-1}[t_0,u](t)
-
\nu
E_{\pa x}^{\beta-\frac{\alpha}{2-\alpha}}[t_0,u](t)
\\
&\quad+
\nu
K_4\left(\beta-\frac{\alpha}{2-\alpha}\right)
E_{\pa t}^{\beta-\frac{\alpha}{2-\alpha}}[t_0,u](t)
+
\nu
K_5\left(\alpha,\beta-\frac{\alpha}{2-\alpha}\right)
E_{a}^{\beta-\frac{2}{2-\alpha}}[t_0,u](t)
\\
&\leq
\left(
\nu
K_4(2-\alpha)^{\frac{2\alpha}{2-\alpha}}
-
\frac{1}{2}\right)
E_{a}^{\beta}[t_0,\pa_t u](t)
+
\left(K_1t_0^{-\frac{2(1-\alpha)}{2-\alpha}}-\nu\right)
E_{\pa x}^{\beta-\frac{\alpha}{2-\alpha}}[t_0,u](t)
+
\nu
K_5
E_{a}^{\beta-\frac{2}{2-\alpha}}[t_0,u](t).
\end{align*}
Taking 
$\nu_\beta^*=4^{-1}(2-\alpha)^{-\frac{2-\alpha}{2-\alpha}}K_4(\beta-\frac{\alpha}{2-\alpha})^{-1}$ and 
$t_4\geq t_3$ such that 
$K_1t_4^{-\frac{2(1-\alpha)}{2-\alpha}}\leq \nu^*_\beta$, we have  the desired estimate.
\end{proof}

\begin{proof}[Proof of Theorem \ref{main2}]
First we note that by Theorem \ref{main1} we have
\begin{align}\label{high.1st}
   \Big(
   E_{\pa x}^{\frac{\gamma}{2-\alpha}}[t_0,u](t)
   +
   E_{\pa t}^{\frac{\gamma}{2-\alpha}}[t_0,u](t)
   \Big)
   + 
\int_{0}^t
E_a^{\frac{\gamma}{2-\alpha}}[t_0,\pa_t u](s)
\,ds\leq 
M_1
\int_{\Omega}
   \Big(|\nabla u_0|^2+|u_1|^2\Big)|x|^{\gamma}
\,dx.
\end{align}
Then applying Lemma \ref{energy1-3} with $\beta=\frac{\gamma+2j}{2-\alpha}$
and $w=\pa_t^{j}u$, we have
\begin{align*}
&\frac{d}{dt}
\Big[
E_{\pa x}^{\frac{\gamma+2j}{2-\alpha}}[t_0,\pa_t^ju](t)
+
E_{\pa t}^{\frac{\gamma+2j}{2-\alpha}}[t_0,\pa_t^ju](t)
+
K_1\Big(
E_{*}^{\frac{\gamma+2j}{2-\alpha}-\frac{\alpha}{2-\alpha}}[t_0,\pa_t^ju](t)
+
E_{a}^{\frac{\gamma+2j}{2-\alpha}-\frac{\alpha}{2-\alpha}}[t_0,\pa_t^ju](t)
\Big)
\Big]
\\
&\leq
-\frac{1}{4}
E_{a}^{\frac{\gamma+2j}{2-\alpha}}[t_0,\pa_t^{j+1}u](t)
+
\widetilde{K}_6\left(\alpha,\frac{\gamma+2j}{2-\alpha}\right)
E_{a}^{\frac{\gamma+2j}{2-\alpha}-\frac{2}{2-\alpha}}[t_0,\pa_t^ju](t)
\\
&\leq
-\frac{1}{4}
E_{a}^{\frac{\gamma+2j}{2-\alpha}}[t_0,\pa_t^{j+1}u](t)
+
\widetilde{K}_6\left(\alpha,\frac{\gamma+2j}{2-\alpha}\right)
E_{a}^{\frac{\gamma+2(j-1)}{2-\alpha}}[t_0,\pa_t^ju](t).
\end{align*}
This gives that there exists a constant $\widetilde{M}_j$ such that
\begin{align}
\nonumber
&E_{\pa x}^{\frac{\gamma+2j}{2-\alpha}}[t_0,\pa_t^ju](t)
+
E_{\pa t}^{\frac{\gamma+2j}{2-\alpha}}[t_0,\pa_t^ju](t)
+
E_{a}^{\frac{\gamma+2j}{2-\alpha}-\frac{\alpha}{2-\alpha}}[t_0,\pa_t^ju](t)
+
\int_{0}^t
   E_{a}^{\frac{\gamma+2j}{2-\alpha}}[t_0,\pa_t^{j+1}u](s)   
\,ds
\\
&\leq 
\label{high.2nd}
\widetilde{M}_j
\left(E_{\pa x}^{\frac{\gamma+2j}{2-\alpha}}[t_0,\pa_t^ju](0)
+
E_{\pa t}^{\frac{\gamma+2j}{2-\alpha}}[t_0,\pa_t^ju](0)
+
E_{a}^{\frac{\gamma+2j}{2-\alpha}-\frac{\alpha}{2-\alpha}}[t_0,\pa_t^ju](0)
+
\int_{0}^t
   E_{a}^{\frac{\gamma+2(j-1)}{2-\alpha}}[t_0,\pa_t^{j}u](s)   
\,ds
\right).
\end{align}
Combining \eqref{high.1st} and \eqref{high.2nd}, we obtain the desired 
inequality. 
\end{proof}

\begin{remark}If $(u_0,u_1)$ has compact supports or 
decays fast enough, then we also can prove that 
\begin{align*}
   \Big(
   E_{\pa x}^{\frac{\gamma}{2-\alpha}+2\ell}[t_0,\pa_t^{\ell}u](t)
   +
   E_{\pa t}^{\frac{\gamma}{2-\alpha}+2\ell}[t_0,\pa_t^{\ell}u](t)
   \Big)
   + 
\int_{0}^t
E_a^{\frac{\gamma}{2-\alpha}+2\ell}[t_0,\pa_t^{\ell+1}u](s) 
\,ds
\leq M
\end{align*}
which is much better than the estimate in Theorem \ref{main2}.
In other words, 
the energy decay estimates for higher order 
derivatives heavily depend on 
the behavior of initial data near spacial infinity. 
We would omit the proof of the estimate mentioned above.
\end{remark}

\section{Diffusion phenomena}

To finish this paper, 
we prove Theorem \ref{main3} which deals with diffusion phenomena for the solution of \eqref{dw} 
with the initial data $(u_0,u_1)$ satisfying 
a compatibility condition of order $1$. 

\begin{lemma}\label{check.domain}
Under the assumption in Theorem \ref{main3},
we have $u(t)\in D(L_*)$ for all $t\geq 0$ 
and 
\[
|x|^{\alpha}\pa_t^2u\in L^\infty(0,\infty;L^2_{d\mu}).
\]
\end{lemma}

\begin{proof}
We note that $u(t)\in D(\mathfrak{a}_*)$ by the inclusion 
$H^1_0(\Omega) \subset D(\mathfrak{a}_*)$. 
Applying Theorems \ref{main1} and \ref{main2} 
with $\ell=1$, we have for $t\geq 0$, 
\[
\int_{\Omega}
   \Big(|\nabla u|^{2}+|\pa_tu|^2\Big)\Psi^{\frac{\gamma}{2-\alpha}}
\,dx
+
\int_{\Omega}
   \Big(|\nabla \pa_t u|^{2}+|\pa_t^2u|^2\Big)\Psi^{\frac{\gamma+2}{2-\alpha}}
\,dx\leq M
\]
and therefore the assumption $\gamma\geq 2-\alpha$ yields that 
for $t\geq 0$, 
\begin{equation}\label{5.2}
\int_{\Omega}
   \Big(|\nabla u|^{2}+|\pa_tu|^2\Big)|x|^{2-\alpha}
\,dx
+
\int_{\Omega}
   \Big(|\nabla \pa_t u|^{2}+|\pa_t^2u|^2\Big)|x|^{4-\alpha}
\,dx\leq M.
\end{equation}
This inequality implies 
\begin{align*}
\|a(x)^{-\frac{1}{2}}\Delta u\|_{L^2(\Omega)}^2
&=
\int_{\Omega}
  |\Delta u|^2
|x|^{\alpha}\,dx
\\
&\leq
\int_{\Omega}
  \Big|\Delta u-|x|^{-\alpha}\pa_tu\big|^2
|x|^{\alpha}\,dx
+
\int_{\Omega}
  |\pa_tu|^2|x|^{-\alpha}\,dx
\\
&\leq
\int_{\Omega}
  |\pa_t^2u|^2
|x|^{\alpha}\,dx
+
\int_{\Omega}
  |\pa_tu|^2|x|^{-\alpha}\,dx
\\
&\leq M'.
\end{align*}
Since $u(t)$ satisfies the Dirichlet boundary condition on $\pa \Omega$, 
by Lemma \ref{domain.op} 
we deduce $u(t)\in D(L_*)$. 
Moreover, by \eqref{5.2} we see that
\begin{align*}
\big\||x|^{\alpha}\pa_t^2u\big\|_{L^2_{d\mu}}^2
=
\int_{\Omega}
|\pa_t^2u|^2|x|^{\alpha}\,dx
=
\int_{\Omega}
|\pa_t^2u|^2|x|^{4-\alpha-(4-2\alpha)}\,dx\leq M''.
\end{align*}
This completes the proof.
\end{proof}

\begin{proof}[Proof of Theorem \ref{main3}]
Applying Theorem \ref{main2} with $k=1$, we have
\begin{gather}
\label{for.dif1}
\int_{0}^t
E_a^{\frac{\gamma}{2-\alpha}}(s,\pa_tu) 
\,ds\leq 
M_0\mathcal{E}_{0}, \quad t\geq 0.
\\
\label{for.dif2}
   E_{\pa t}^{\frac{\gamma+2}{2-\alpha}}(t,\pa_tu)
\leq 
M_1(\mathcal{E}_0+\mathcal{E}_1), \quad t\geq 0.
\\
\label{for.dif3}
\int_{0}^t
E_a^{\frac{\gamma+2}{2-\alpha}}(s,\pa_t^2u) 
\,ds\leq 
M_1(\mathcal{E}_0+\mathcal{E}_1), \quad t\geq 0.
\end{gather}
Here we rewrite \eqref{dw} as the problem 
\[
\pa_tu-|x|^{\alpha}\Delta u=-|x|^{\alpha}\pa_t^2u, \quad t\geq 0.
\]
with $u(0)=u_0\in L^2_{d\mu}$, where we regard $-|x|^{\alpha}\pa_t^2u$ 
as a inhomogeneous term for the heat equation 
$\pa_tu-|x|^{\alpha}\Delta u=0$ in $L^2_{d\mu}$. 
By virtue of Lemma \ref{check.domain}, we see from the 
standard semigroup theory that 
\[
u(t)=e^{tL*}u_0-\int_{0}^{t}e^{(t-s)L_*}[|x|^{\alpha}\pa_t^2u(s)]\,ds.
\]
As in \cite{Wa14} (see also \cite{So_Wa1} and \cite{So_Wa2}), we have 
\begin{align*}
u(t)
-e^{tL_*}[u_0+|x|^{\alpha}u_1]
&=-
\int_{\frac{t}{2}}^{t}
   e^{(t-s)L_*}\big[|x|^{\alpha}\pa_t^2u(s)\big]
\,ds
-e^{\frac{t}{2}L_*}\big[|x|^{\alpha}\pa_tu(t/2)\big]
\\
&\quad-
\int_0^{\frac{t}{2}}
   L_*e^{(t-s)L_*}\big[|x|^{\alpha}\pa_tu(s)\big]
\,ds.
\end{align*}
Taking the $L^2_{d\mu}$-norm, we see
\begin{align*}
\Big\|u(t)-e^{tL_*}[u_0+|x|^{\alpha}u_1]\Big\|_{L^2_{d\mu}}
&\leq 
\int_{\frac{t}{2}}^t\big\||x|^{\alpha}\pa_t^2(s)\big\|_{L^2_{d\mu}}\,ds
+
\big\||x|^{\alpha}\pa_tu(t/2)\big\|_{L^2_{d\mu}}
\\
&\quad+
\int_{0}^{\frac{t}{2}}
\big\|L_*e^{(t-s)L_*}[|x|^{\alpha}\pa_tu(s)]\big\|_{L^2_{d\mu}}\,ds
\\
&=:
\mathcal{J}_1
+\mathcal{J}_2
+\mathcal{J}_3.
\end{align*}
Schwarz's inequality and 
the definition of $\Psi^\beta$ yield
\begin{align*}
\mathcal{J}_1^2
&\leq 
\frac{t}{2}
\int_{\frac{t}{2}}^t
\big\||x|^{\alpha}\partial_t^2 u(s)\big\|_{L^2_{d\mu}}^2\,ds
\\
&\leq 
\frac{t}{2}(2-\alpha)^{\frac{4\alpha}{2-\alpha}}
\int_{\frac{t}{2}}^t
\left(\int_{\Omega}|x|^{-\alpha}|\pa_t^2u(s)|^2\Psi^{\frac{2\alpha}{2-\alpha}}\,dx\right)
\,ds
\\
&\leq 
\frac{t}{2}(2-\alpha)^{\frac{4\alpha}{2-\alpha}}
\int_{\frac{t}{2}}^t
(t_0+s)^{\frac{2\alpha}{2-\alpha}-\frac{\gamma+2}{2-\alpha}}
\left(\int_{\Omega}|x|^{-\alpha}|\pa_t^2u(s)|^2\Psi^{\frac{\gamma+2}{2-\alpha}}\,dx\right)
\,ds
\\
&\leq
\frac{t}{2}(2-\alpha)^{\frac{4\alpha}{2-\alpha}}
\left(t_0+\frac{t}{2}\right)^{-\frac{\gamma-\alpha}{2-\alpha}-1}
\int_{\frac{t}{2}}^t
E_{a}^{\frac{\gamma+2}{2-\alpha}}[t_0,\pa_t^2u](s)
\,ds.
\end{align*}
By \eqref{for.dif3}, we have 
$\mathcal{J}_1\leq 
\mathcal{K}_1(1+t)^{-\frac{\gamma-\alpha}{2(2-\alpha)}}(\mathcal{E}_0+\mathcal{E}_1)^{\frac{1}{2}}$.
By a computation similar to the one for $\mathcal{J}_1$, 
we deduce from \eqref{for.dif2} that 
\begin{align*}
\mathcal{J}_2^2
&=
\int_{\Omega}|x|^{\alpha}|\pa_tu(t/2)|^2\,dx
\\
&\leq 
(2-\alpha)^{\frac{4\alpha}{2-\alpha}}
\int_{\Omega}
   |x|^{-\alpha}|\pa_tu(t/2)|^2
   \Psi^{\frac{2\alpha}{2-\alpha}}
\,dx
\\
&\leq 
(2-\alpha)^{\frac{4\alpha}{2-\alpha}}
\left(t_0+\frac{t}{2}\right)^{-\frac{\gamma-\alpha}{2-\alpha}}
\int_{\Omega}
   |x|^{-\alpha}|\pa_tu(t/2)|^2
   \Psi^{\frac{\gamma+\alpha}{2-\alpha}}
\,dx
\\
&\leq 
(2-\alpha)^{\frac{4\alpha}{2-\alpha}}
\left(t_0+\frac{t}{2}\right)^{-\frac{\gamma-\alpha}{2-\alpha}}
E_a^{\frac{\gamma+2}{2-\alpha}-1}
[t_0,\pa_tu](t/2)
\\
&\leq \mathcal{K}_2^2
(1+t)^{-\frac{\gamma-\alpha}{2-\alpha}}(\mathcal{E}_0+\mathcal{E}_1).
\end{align*}
For $\mathcal{J}_3$, we divide the proof 
into two cases 	
$2-\alpha\leq \gamma\leq N+\alpha$
and $N+\alpha<\gamma<N+2-\alpha$. 
In the former case, we set $1<p(\theta)\leq 2$ $(\theta>0)$ as
\[
p(\theta)
= 
\frac{2(N-\alpha+\theta)}{N+\gamma-3\alpha+\theta}
>
1.
\]
that is, 
\[
p(\theta)\alpha-\alpha = 
(\gamma-\alpha)\frac{p(\theta)}{2}
-(N+\theta)\left(1-\frac{p(\theta)}{2}\right).
\]
Then we have 
\begin{align*}
\big\|
  |x|^{\alpha}\pa_tu(s)\big
\|_{L^{p(\theta)}_{d\mu}}^2
&=
\left(
   \int_{\Omega}|x|^{p(\theta)\alpha-\alpha}|\pa_tu(s)|^{p(\theta)}\,dx
\right)^{\frac{2}{p(\theta)}}
\\
&\leq 
\left(
   \int_{\Omega}|x|^{-N-\theta}\,dx
\right)^{\frac{2}{p(\theta)}-1}
   \int_{\Omega}|x|^{\gamma-\alpha}|\pa_tu(s)|^{2}\,dx
\\
&\leq 
(2-\alpha)^{\frac{2\gamma}{2-\alpha}}
\left(
   \int_{\Omega}|x|^{-N-\theta}\,dx
\right)^{\frac{2}{p(\theta)}-1}
\int_{\Omega}
   |\pa_tu(s)|^{2}|x|^{-\alpha}
   \Psi^{\frac{\gamma}{2-\alpha}}\,dx
\\
&\leq 
(2-\alpha)^{\frac{2\gamma}{2-\alpha}}
\left(
   \int_{\Omega}|x|^{-N-\theta}\,dx
\right)^{\frac{2}{p(\theta)}-1}
 E_{a}^{\frac{\gamma}{2-\alpha}}[t_0,\pa_tu](s).
\end{align*}
And then by \eqref{for.dif1}, 
\begin{align*}
\mathcal{J}_3
&=
\int_{0}^\frac{t}{2}
\big\|L_*e^{(t-s)L_*}[|x|^{\alpha}\partial_t u(s)]\big\|_{L^2_{d\mu}}
\,ds
\\
&\leq 
\int_{0}^\frac{t}{2}
(t-s)^{-\frac{N-\alpha}{2-\alpha}(\frac{1}{p(\theta)}-\frac{1}{2})-1}
\big\||x|^{\alpha}\partial_t u(s)\big\|_{L^{p(\theta)}_{d\mu}}
\,ds
\\
&\leq 
\left(\frac{t}{2}\right)^{-\frac{N-\alpha}{2-\alpha}
(\frac{1}{p(\theta)}-\frac{1}{2})-\frac{1}{2}}
\left(\int_{0}^\frac{t}{2}
\big\||x|^{\alpha}\partial_tu(s)\big\|_{L^{p(\theta)}_{d\mu}}^2
\,ds\right)^\frac{1}{2}
\\
&\leq 
\mathcal{K}_{3,\theta}(1+t)^{-\frac{N-\alpha}{2-\alpha}(\frac{1}{p(\theta)}-\frac{1}{2})-\frac{1}{2}}\mathcal{E}_0^{\frac{1}{2}}.
\end{align*}
Noting that by choosing $\theta$ small enough, 
\begin{align*}
\frac{N-\alpha}{2-\alpha}
\left(\frac{1}{p(\theta)}-\frac{1}{2}\right)+\frac{1}{2}
&=
\frac{(\gamma-2\alpha)(N-\alpha)}{2(N+\theta-\alpha)(2-\alpha)}+\frac{1}{2}
\\
&=
\frac{\gamma-\alpha}{2(2-\alpha)}
+
\frac{1-\alpha}{2-\alpha}
-
\frac{(\gamma-2\alpha)\theta}{2(N+\theta-\alpha)(2-\alpha)}
\\
&\geq 
\frac{\gamma-\alpha}{2(2-\alpha)}, 
\end{align*}
we have $\mathcal{J}_3\leq \mathcal{K}_{3,\theta}
(1+t)^{-\frac{\gamma-\alpha}{2-\alpha}}$. 
In the latter case, we choose $p=1$ and then we deduce 
\begin{align*}
\big\|
  |x|^{\alpha}\pa_tu(s)\big
\|_{L^1_{d\mu}}^2
&=
\left(
   \int_{\Omega}|\pa_tu(s)|\,dx
\right)^{2}
\\
&\leq 
\left(
   \int_{\Omega}|x|^{\alpha-\gamma}\,dx
\right)
   \int_{\Omega}|x|^{\gamma-\alpha}|\pa_tu(s)|^{2}\,dx
\\
&\leq 
(2-\alpha)^{\frac{2\gamma}{2-\alpha}}
\left(
   \int_{\Omega}|x|^{\alpha-\gamma}\,dx
\right)
 E_{a}^{\frac{\gamma}{2-\alpha}}[t_0,\pa_tu](s).
\end{align*}
This implies that 
\begin{align*}
\mathcal{J}_3
&=
\int_{0}^\frac{t}{2}
\big\|L_*e^{(t-s)L_*}[|x|^{\alpha}\partial_tu(s)]\big\|_{L^2_{d\mu}}
\,ds
\\
&\leq 
\int_{0}^\frac{t}{2}
(t-s)^{-\frac{N-\alpha}{2(2-\alpha)}-1}
\big\||x|^{\alpha}\partial_tu(s)\big\|_{L^1_{d\mu}}
\,ds
\\
&\leq 
\left(\frac{t}{2}\right)^{-\frac{N-\alpha}{2(2-\alpha)}-\frac{1}{2}}
\left(\int_{0}^\frac{t}{2}
E_a^{\frac{\alpha}{2-\alpha}}[t_0, \partial_tu](s)
\,ds\right)^\frac{1}{2}
\\
&\leq 
\mathcal{K}_{3}'(1+t)^{-\frac{N-\alpha}{2(2-\alpha)}-\frac{1}{2}}\mathcal{E}_0^{\frac{1}{2}}.
\end{align*}
Since 
\[
\frac{N-\alpha}{2(2-\alpha)}+\frac{1}{2}
=
\frac{N+2-2\alpha}{2(2-\alpha)}
>\frac{\gamma}{2(2-\alpha)}
\geq\frac{\gamma-\alpha}{2(2-\alpha)}
\] 
by the assumption of this case, we have 
$\mathcal{J}_3\leq \mathcal{K}_{3}'
(1+t)^{-\frac{\gamma-\alpha}{2-\alpha}}\mathcal{E}_0^{\frac{1}{2}}$. 
Consequently, in both case we obtain 
\[
\Big\|u(t)-e^{tL}[u_0+|x|^{\alpha}u_1]\Big\|_{L^2_{d\mu}}
\leq 
\mathcal{K}_0(1+t)^{-\frac{\gamma-\alpha}{2(2-\alpha)}}(\mathcal{E}_0+\mathcal{E}_1)^{\frac{1}{2}}
\]
with some constant $\mathcal{K}_0>0$. 
\end{proof}

\section*{Appendix}
\renewcommand{\thesection}{A}
\setcounter{theorem}{0}
\setcounter{equation}{0}

To reader's convenience, 
we collect some properties of  solutions to 
Kummer's confluent hypergeometric differential equation
\begin{equation}\label{kummer.eq}
su''(s)+(c-s)u'(s)-bu(s)=0, \quad s\in (0,\infty).
\end{equation}
(see e.g., Beals--Wong \cite{BW}). 

\begin{definition}\label{def.kummer}
Let $b,c\in\R$ satisfy $-c\notin \N\cup \{0\}$. Then define 
Kummer's confluent hypergeometric function of first kind 
as follows:
\[
M(b,c;s)=
\sum_{n=0}^{\infty}\frac{b_n}{c_n}\,\frac{s^n}{n!}, \quad s\in (0,\infty), 
\]
where $b_0=c_0=1$, $b_n=\prod_{k=1}^n(b+k-1)$ and $c_n=\prod_{k=1}^n(c+k-1)$ for $n\in \N$. 
For $b>0$, we also define 
Kummer's confluent hypergeometric function of second kind
as follows
\[
U(b,c;s)=\frac{1}{\Gamma(b)}
\int_{0}^\infty
  e^{-\sigma s}\sigma^{b-1}(1+\sigma)^{c-b-1}\,d\sigma, \quad s\in (0,\infty).
\]
\end{definition}

\begin{remark}
If $b$ is a negative integer, then 
$b_n=0$ for every $n>-b$. Therefore $M(b,c,s)$ is a polynomial of degree $-b$.
The function $M(b,c;s)$ tends to $1$ as $s\to 0$
and the function $U(b,c;s)$ tends to $\infty$ as $s\to 0$.
\end{remark}

The following lemma is a collection of properties 
stated in \cite{BW}, which are needed in the present paper. 
\begin{lemma}\label{Kummer}
The following three assertions hold:
\begin{itemize}
\item[\bf (i)]\ The pair $(M(b,c;\cdot), U(b,c;\cdot))$ 
is a fundamental system of the equation \eqref{kummer.eq}. 
\item[\bf (ii)]\ If $c>b>0$, then 
$M(b,c;s)$ and satisfies 
$M(b,c;s)\sim \frac{\Gamma(c)}{\Gamma(b)}s^{b-c}e^s $ 
as $s\to \infty$, more precisely, 
\[
\lim_{s\to \infty}
\left(\frac{M(b,c;s)}{s^{b-c}e^s}\right)=\frac{\Gamma(c)}{\Gamma(b)}.
\]
\item[\bf (iii)]\ If $b>0$, then 
$U(b,c;s)$ satisfies 
$U(b,c;s)\sim s^{-b}$ as $s\to \infty$, more precisely, 
\[
\lim_{s\to \infty}
\left(\frac{U(b,c;s)}{s^{-b}}\right)=1.
\]
\end{itemize}
\end{lemma}


\section*{Acknowledgments}

This work is supported by 
Grant-in-Aid for JSPS Fellows 15J01600 
of Japan Society for the Promotion of Science
and 
also 
partially supported 
by Grant-in-Aid for Young Scientists Research (B), 
No.\ 16K17619.



\begin{thebibliography}{30}

\bibitem{AGG06}
    W. Arendt, R.G.Goldstein, J.A.Goldstein, 
    Outgrowths of Hardy's inequality,
    \emph{Recent advances in differential equations and mathematical physics}, 
    51--68, \emph{Contemp. Math.}, {\bf 412}, Amer. Math. Soc., Providence, RI, 2006. 

\bibitem{BW}
    R. Beals, R. Wong, 
    ``Special functions,''
    A graduate text. Cambridge Studies in Advanced Mathematics {\bf 126}, 
    Cambridge University Press, Cambridge, 2010.
 
\bibitem{ChHa03}
R. Chill, A. Haraux,
An optimal estimate for the difference of solutions of two abstract evolution equations,
J. Differential Equations {\bf 193} (2003), 385--395.


\bibitem{HoOg04}
    T. Hosono, T. Ogawa,
    Large time behavior and $L^p$-$L^q$ estimate of
solutions of 2-dimensional nonlinear damped wave equations,
J. Differential Equations {\bf 203} (2004), 82--118.


\bibitem{Ik68} 
    M. Ikawa,
    Mixed problems for hyperbolic equations of second order,
    \emph{J. Math.\ Soc.\ Japan} {\bf 20} (1968), 580--608.

\bibitem{Ikbook} 
    M. Ikawa,
    Hyperbolic partial differential equations and wave phenomena,
    American Mathematical Society (2000).
    
 \bibitem{Ik02}
 R. Ikehata,
 Diffusion phenomenon for linear dissipative wave equations in an exterior domain,
J. Differential Equations {\bf 186} (2002), 633--651.


\bibitem{Ik05IJPAM} 
    R. Ikehata,
    Some remarks on the wave equation with potential type damping coefficients,
    \emph{Int.\ J. Pure Appl.\ Math.}\ {\bf 21} (2005), 19--24.

\bibitem{IT05}
    R. Ikehata, K. Tanizawa, 
    Global existence of solutions for semilinear damped wave equations 
    in $\R^N$ with noncompactly supported initial data, 
    \emph{Nonlinear Anal.} {\bf 61} (2005), 1189--1208.
    
\bibitem{Kar00}
    G. Karch,
    Selfsimilar profiles in large time asymptotics of solutions to damped wave equations,
    Studia Math.\ {\bf 143} (2000), 175--197.
    
 \bibitem{LNZ10}
    J. Lin, K. Nishihara, J. Zhai, 
    $L^2$-estimates of solutions for damped wave equations 
    with space-time dependent damping term, 
    \emph{J.\ Differential Equations} {\bf 248} (2010), 403--422.


\bibitem{MaNi03}
     P. Marcati, K. Nishihara,
     The $L^p$-$L^q$ estimates of solutions to
one-dimensional damped wave equations and
their application to the compressible flow through porous media,
J. Differential Equations {\bf 191} (2003), 445--469.

\bibitem{Ma76}
    A. Matsumura,
    On the asymptotic behavior of solutions of semi-linear wave equations,
    \emph{Publ.\ Res.\ Inst.\ Math.\ Sci.}\ {\bf 12} (1976), 169--189.

\bibitem{Ma77}
    A. Matsumura,
    Energy decay of solutions of dissipative wave equations,
    \emph{Proc.\ Japan Acad., Ser.\ A} {\bf 53} (1977), 232--236.

\bibitem{Mitidieri00}
    Mitidieri, E., 
    A simple approach to Hardy inequalities (Russian) 
    {\emph Mat.\ Zametki} {\bf 67} (2000), 563--572; 
    translation in \emph{Math.\ Notes} {\bf 67} (2000), 479--486.

\bibitem{Mo76} 
    K. Mochizuki,
    Scattering theory for wave equations with dissipative terms,
    \emph{Publ. Res. Inst. Math. Sci.} {\bf 12} (1976), 383--390.
    
\bibitem{Na04}
    T. Narazaki,
    $L^p$-$L^q$ estimates for damped wave equations
and their applications to semi-linear problem,
J. Math.\ Soc.\ Japan {\bf 56} (2004), 585--626.


\bibitem{Nishihara03}
    K. Nishihara, 
    $L^p$-$L^q$ estimates of solutions to the damped wave equation 
    in $3$-dimensional space and their application, 
    \emph{Math. Z.} {\bf 244} (2003), 631--649.

\bibitem{Nishihara10}
    K. Nishihara, 
    Decay properties for the damped wave equation 
    with space dependent potential and absorbed semilinear term, 
    \emph{Comm.\ Partial Differential Equations} {\bf 35} (2010), 1402--1418.

\bibitem{RTY09}
    P. Radu, G. Todorova, and B. Yordanov, 
    Higher order energy decay rates for damped wave equations 
    with variable coefficients, 
    \emph{Discrete Contin.\ Dyn.\ Syst.\ Ser.\ S}, {\bf 2}\ (2009), 609--629.

\bibitem{RTY10}
    P. Radu, G. Todorova, and B. Yordanov, 
    Decay estimates for wave equations with variable coefficients, 
    \emph{Trans.\ Amer.\ Math.\ Soc.}, {\bf 362}\ (2010), 2279--2299. 

\bibitem{Ikehata00} 
    A. Saeki, R. Ikehata, 
    Remarks on the decay rate for the energy of the dissipative linear wave equations in exterior domains, 
    \emph{SUT\ J.\ Math.\ } {\bf 36} (2000), 267--277.

\bibitem{So_Wa1}
    M. Sobajima and Y. Wakasugi, 
    Diffusion phenomena for the wave equation 
    with space-dependent damping in an exterior domain,
    \emph{J. Differential Equations}, 
    {\bf 261} (2016), 5690--5718. 

\bibitem{So_Wa2}
    M. Sobajima and Y. Wakasugi, 
    Remarks on an elliptic problem arising 
    in weighted energy estimates 
    for wave equations 
    with space-dependent damping term 
    in an exterior domain, 
    \emph{AIMS Mathematics}, 
    {\bf 2} (2017), 1--15. 

\bibitem{ToYo01}
    G. Todorova, B. Yordanov,
    Critical exponent for a nonlinear wave equation with damping,
    \emph{J. Differential Equations} {\bf 174} (2001), 464--489.

\bibitem{ToYo09} 
    G. Todorova, B. Yordanov,
    Weighted $L^2$-estimates for dissipative wave equations with
    variable coefficients,
    \emph{J. Differential Equations} {\bf 246} (2009), 4497--4518.

\bibitem{Wa12}
    Y. Wakasugi, 
    Small data global existence for the semilinear wave equation 
    with space-time dependent damping, 
    \emph{J.\ Math.\ Anal.\ Appl.} {\bf 393} (2012), 66--79.
\bibitem{Wa14}
    Y. Wakasugi,
    On diffusion phenomena for the linear wave equation with space-dependent damping,
    \emph{J.\ Hyp.\ Diff.\ Eq.}\ {\bf 11} (2014), 795--819.
    
\bibitem{YaMi00}
H. Yang, A. Milani,
On the diffusion phenomenon of quasilinear hyperbolic waves,
Bull. Sci. Math. {\bf124} (2000), 415--433.
\end{thebibliography}
\end{document}